\documentclass[11pt]{amsart}
\usepackage{amsfonts}
\usepackage{amsmath}
\usepackage{amsthm}
\usepackage{epsfig}
\usepackage{epstopdf}

\usepackage{url}
\usepackage[all]{xy}
\usepackage{latexsym}
\usepackage{amssymb}
\usepackage[cp850]{inputenc}
\usepackage{amsfonts}
\usepackage{graphicx}
  \usepackage{etex}
  \usepackage{xcolor}
  \definecolor{verydarkblue}{rgb}{0,0,0.5}
  \usepackage[
      breaklinks,
      colorlinks,
      citecolor=verydarkblue,
      linkcolor=verydarkblue,
      urlcolor=verydarkblue,
      pagebackref=true,
      hyperindex
  ]{hyperref}
\newtheorem{theorem}{Theorem}[section]		
\newtheorem{lemma}[theorem]{Lemma}
			
\newtheorem{prop}[theorem]{Proposition}

\newtheorem*{defi*}{Definition}			
\newtheorem*{bei*}{Example}
\newtheorem*{sat*}{Theorem}				
\newtheorem*{kor*}{Corollary}
\newtheorem*{rmk*}{Remark}				
	
\newtheorem*{quest*}{Question}	
	
\newtheorem{claim}{Claim}	

\let\ssection=\section
\renewcommand{\section}{\setcounter{equation}{0}\ssection}

\newtheorem*{namedtheorem}{\theoremname}
\newcommand{\theoremname}{testing}
\newenvironment{named}[1]{\renewcommand{\theoremname}{#1}\begin{namedtheorem}}{\end{namedtheorem}}

\theoremstyle{remark}
\newtheorem*{bem}{Remark}

\newtheorem*{namedtheoremr}{\theoremnamer}
\newcommand{\theoremnamer}{testing}

			\newcommand{\BH}{\mathbb H}
\newcommand{\BR}{\mathbb R}			
			
\newcommand{\BS}{\mathbb S}			\newcommand{\BZ}{\mathbb Z}

\newcommand{\BP}{\mathbb P}

\newcommand{\CC}{\mathcal C}		
\newcommand{\CE}{\mathcal E}		\newcommand{\CF}{\mathcal F}
\newcommand{\CG}{\mathcal G}		
		
		\newcommand{\CL}{\mathcal L}
\newcommand{\CM}{\mathcal M}		\newcommand{\CN}{\mathcal N}
		\newcommand{\CP}{\mathcal P}
\newcommand{\CQ}{\mathcal Q}		
\newcommand{\CS}{\mathcal S}		\newcommand{\CT}{\mathcal T}
\newcommand{\CU}{\mathcal U}

\newcommand{\FM}{\mathfrak m}

\newcommand{\FB}{\mathfrak b}
\newcommand{\FC}{\mathfrak c}

\newcommand{\actson}{\curvearrowright}
\newcommand{\D}{\partial}

\DeclareMathOperator{\Diff}{Diff}	%	Diffeomorphimen einer Mf
\DeclareMathOperator{\PSL}{PSL}		%	Spezielle lineare Gruppe
\DeclareMathOperator{\Id}{Id}		%	Identit\"at
\DeclareMathOperator{\Isom}{Isom}	%	Isometrien einer Mf
\DeclareMathOperator{\Map}{Map}
\DeclareMathOperator{\PMap}{PMap}

\DeclareMathOperator{\Ker}{Ker}

\newcommand{\comment}[1]{}

\DeclareMathOperator{\Stab}{Stab}

\DeclareMathOperator{\Homeo}{Homeo}

\DeclareMathOperator{\hyp}{hyp}
\DeclareMathOperator{\Thu}{Thu}

\DeclareMathOperator{\sing}{sing}

\DeclareMathOperator{\orb}{or}

\DeclareMathOperator{\RO}{O}

\begin{document}

\title[]{Counting curves on orbifolds}
  \author{Viveka Erlandsson}
  \address{School of Mathematics, University of Bristol \\ Bristol BS8 1UG, UK {\rm and}  \newline ${ }$ \hspace{0.2cm} Department of Mathematics and Statistics, UiT The Arctic University of Norway}
  \email{v.erlandsson@bristol.ac.uk}
\thanks{The first author gratefully acknowledges support from EPSRC grant EP/T015926/1.}
\author{Juan Souto}
\address{UNIV RENNES, CNRS, IRMAR - UMR 6625, F-35000 RENNES, FRANCE}
\email{jsoutoc@gmail.com}

\begin{abstract}
We show that Mirzakhani's curve counting theorem also holds if we replace surfaces by orbifolds.  
\end{abstract}
\maketitle

\section{Introduction}

Throughout this paper we let $\Gamma\subset\PSL_2\BR$ be a non-elementary finitely generated discrete subgroup of the group of orientation preserving isometries of the hyperbolic plane $\BH^2$. Suppose for a moment that $\Gamma$ is torsion free, let $S=\BH^2/\Gamma$ be the associated hyperbolic surface, and let $\CS(S)$ be the set of free homotopy classes of closed unoriented primitive essential curves therein. Here essential just means that the given homotopy class of curves is neither trivial nor peripheral. The mapping class group $\Map(S)$ of $S$ acts on $\CS(S)$ and we say that two elements in the same orbit are {\em of the same type}. Mirzakhani studied the asymptotic behavior, when $L\to\infty$, of the number of elements in $\CS(S)$ of some fixed type $\gamma_0$ and with at most length $L$. More concretely, she proved in \cite{Maryam0,Maryam1,Maryam2} that the limit
\begin{equation}\label{eq-maryam}
\lim_{L\to\infty}\frac 1{L^{6g-6+2r}}\vert\{\gamma\text{ of type }\gamma_0\text{ with }\ell_S(\gamma)\le L\}\vert
\end{equation}
exists and is positive for every $\gamma_0\in\CS(S)$. Here, $g$ is the genus of $S$, $r$ is the number of ends, and $\ell_S(\gamma)$ is the length of the hyperbolic geodesic in the homotopy class $\gamma$. 

The goal of this note is to prove that this statement remains true when $\Gamma$ has torsion, that is when $\RO=\BH^2/\Gamma$ is an orbifold instead of a surface. 

\begin{theorem}\label{sat1}
Let $\Gamma\subset\PSL_2\BR$ be a non-elementary finitely generated discrete subgroup and $\RO=\BH^2/\Gamma$ the associated 2-dimensional hyperbolic orbifold. Then the limit
$$\lim_{L\to\infty}\frac 1{L^{6g-6+2r}}\vert\{\gamma\text{ of type }\gamma_0\text{ with }\ell_{\RO}(\gamma)\le L\}\vert$$
exists and is positive for any $\gamma_0\in\CS^{\orb}(\RO)$. Here $g$ is the genus of the orbifold $\RO$ and $r$ is the sum of the numbers of singular points and ends.
\end{theorem}

A few comments on the notation and terminology used in Theorem \ref{sat1}:
\medskip

\noindent{\bf (1)} The topological space underlying an orientable 2-dimensional hyperbolic orbifold is an orientable topological surface. The genus of the orbifold is by definition the genus of that surface.

\noindent{\bf (2)} In the theorem, and also in the remaining of the paper, $\CS^{\orb}(\RO)$ is the set of free homotopy classes of closed unoriented primitive essential curves, where the homotopy is taken in the category of orbifolds and where essential means that the curves in the given homotopy class are neither peripheral nor represent finite order elements in the orbifold fundamental group $\pi_1^{\orb}(\RO)$. Accordingly, $\ell_{\RO}(\gamma)$ is the length of the shortest curve homotopic to $\gamma$ in the category of orbifolds. See Section \ref{subsec:maps between orbifolds} for details.

\noindent{\bf (3)} As for surfaces, two elements in $\CS^{\orb}(\RO)$ are of the same type if they differ by an element of the mapping class group
$$\Map^{\orb}(\RO)=\Homeo^{\orb}(\RO)/\Homeo^{\orb}_0(\RO).$$
Here $\Homeo^{\orb}(\RO)$ is the group of homeomorphisms of $\RO$ in the category of orbifolds, and $\Homeo^{\orb}_0(\RO)$ is its identity component. The mapping class group $\Map^{\orb}(\RO)$ is infinite unless $\RO$ is {\em exceptional}, by what we mean that it has genus $g=0$ and that $r=3$. See Section \ref{sec orbifold mapping class group} for more details on the mapping class group of an orbifold.
\medskip

As is already the case for the proof in \cite{Book} of Mirzakhani's \eqref{eq-maryam}, we will derive Theorem \ref{sat1} from the weak-*-convergence of certain measures on the space $\CC^{\orb}(\RO)$ of currents, that is the space of $\pi_1^{\orb}(\RO)$-invariant Radon measures on the set of geodesics on the orbifold universal cover $\tilde\RO$ of $\RO$. Trusting that the reader is familiar with currents, we just recall at this point that  the set $\BR_{\ge 0}\CS^{\orb}(\RO)$ of weighted curves is a dense subset of $\CC^{\orb}(\RO)$, that $\CC^{\orb}(\RO)$ is a cone in a linear space, and that the action of $\Map^{\orb}(\RO)$ on $\CS^{\orb}(\RO)$ extends to a linear action on $\CC^{\orb}(\RO)$. We will recall a few facts about currents in Section \ref{sec currents} below, but we do already at this point refer the reader to \cite{AL,BonahonFrench,Bonahon2,Bonahon,Book} for details and background.

\begin{theorem}\label{main}
Let $\RO$ be a compact orientable non-exceptional hyperbolic orbifold with possibly empty totally geodesic boundary and let $\CC^{\orb}(\RO)$ be the associated space of geodesic currents. There is a Radon measure $\FM_{\Thu}$ on $\CC^{\orb}(\RO)$ such that for any $\gamma_0\in\CS^{\orb}(\RO)$ we have
$$\lim_{L\to\infty}\frac 1{L^{6g-6+2r}}\sum_{\gamma\in\Map^{\orb}(\RO)\cdot\gamma_0}\delta_{\frac 1L\gamma}=C(\gamma_0)\cdot\FM_{\Thu}$$
for some positive constant $C(\gamma_0)>0$. Here $g$ is the genus of the orbifold $\RO$, $r$ is the sum of the numbers of singular points and boundary components, and $\RO$ is non-exceptional if $(g,r)\neq(0,3)$. Moreover $\delta_{\frac 1L\gamma}$ stands for the Dirac measure on $\CC^{\orb}(\RO)$ centered at $\frac 1L\gamma$, and the convergence takes place with respect to the weak-*-topology on the space of Radon measures on $\CC^{\orb}(\RO)$.
\end{theorem}

Again a few comments:
\medskip

\noindent{\bf (1)} Theorem \ref{main} remains true if we replace curves by multicurves, that is if we replace $\gamma_0$ by finite formal linear combinations (with positive coefficients) of elements in $\CS^{\orb}(\RO)$. In fact, the proof is just the same, only needing a bit more of notation to keep track of things, and the interested reader will have no difficulties making the necessary tweaks.

\noindent{\bf (2)} Also, as is the case for surfaces, the statement of Theorem \ref{main} remains true if we replace $\Map^{\orb}(\RO)$ by a finite index subgroup $G$, and the constant on the right side changes exactly as it does in the case of surfaces---see \cite[Exercise 8.2]{Book}. In fact, in the course of the proof of Theorem \ref{main} we will have to work with such a finite index subgroup, the pure mapping class group.

\noindent{\bf (3)} The measure $\FM_{\Thu}$ in the statement of Theorem \ref{main} arises as the push-forward under a certain map of the usual Thurston measure on the space of measured laminations of a surface. We will however also give a short intrinsic description of $\FM_{\Thu}$ in Section \ref{sec thurston measure} below.

\noindent{\bf (4)} If one were to drop the assumption in Theorem \ref{main} that the orbifold is non-exceptional then the limit would trivially exist because the mapping class group would be finite, but the measure class of the obtained measure would obviously depend on $\gamma_0$. This is why we do need this assumption in Theorem \ref{main} but not in Theorem \ref{sat1} above or in Theorem \ref{sat3} below.
\medskip

All of this is nice and well and cute, but a more substantial observation is that Theorem \ref{main} implies that Theorem \ref{sat1} also holds if we replace $\ell_{\RO}$ by many other notions of length: length with respect to a variable curvature metric, word-length, extremal length, and so on. In fact, we can replace $\ell_{\RO}$ by any continuous homogenous function
$$F:\CC^{\orb}(\RO)\to\BR_{\ge 0}$$
on the space of currents, where homogeneous means that $F(t\cdot\lambda)=t\cdot F(\lambda)$. See \cite{EPS,DidacThurston} for many examples of such functions. 

\begin{theorem}\label{sat3}
Let $\RO$ be a compact orientable hyperbolic orbifold with possibly empty totally geodesic boundary and let $\CC^{\orb}(\RO)$ be the associated space of geodesic currents. Then the limit
$$\lim_{L\to\infty}\frac 1{L^{6g-6+2r}}\vert\{\gamma\text{ of type }\gamma_0\text{ with }F(\gamma)\le L\}\vert$$
exists and is positive for any $\gamma_0\in\CS^{\orb}(\RO)$ and any positive, homogenous, continuous function $F:\CC^{\orb}(\RO)\to\BR_{\ge 0}$. Here $g$ is the genus of the orbifold $\RO$, $r$ is the sum of the numbers of singular points and boundary components. 
\end{theorem}

Let us now describe the strategy of the proof of our main result, Theorem \ref{main}. Instead of aiming to give a stand alone proof of the theorem along the lines of the proof in \cite{Book} of the corresponding result for surfaces, we are going to use the latter to obtain that for orbifolds. Suppose for the sake of concreteness that $\RO$ has no boundary and a single cone point and let $\Sigma$ be the surface obtained by deleting from $\RO$ a small ball around that singular point. The inclusion $\Sigma\hookrightarrow\RO$ induces a surjective map 
\begin{equation}\label{eq silly map sick of this}
\CS(\Sigma)\to\CS(\RO)\cup\{*\}
\end{equation}
where $*$ is just a point where one maps all essential curves in $\Sigma$ which are non-essential in $\RO$. In fact, this map is equivariant under the isomorphism $\Map(\Sigma)\simeq\Map^{\orb}(\RO)$ between the corresponding mapping class groups. Equivariance under this isomorphism implies that whenever $\eta_0\in\CS(\Sigma)$ maps to $\gamma_0\in\CS^{\orb}(\RO)$ then the push-forward under \eqref{eq silly map sick of this} of the measure $\frac 1{L^{6g-6+2r}}\sum_{\eta\in\Map(S)\cdot\eta_0}\delta_{\frac 1L\eta}$ is the measure $\frac 1{L^{6g-6+2r}}\sum_{\gamma\in\Map^{\orb}(\RO)\cdot\gamma_0}\delta_{\frac 1L\gamma}$. From the analogue of Theorem \ref{main} for surfaces, stated in Section \ref{need batteries} below, we get that the limit 
$$\lim_{L\to\infty}\frac 1{L^{6g-6+2r}}\sum_{\eta\in\Map(S)\cdot\eta_0}\delta_{\frac 1L\eta}$$ 
exists. As we see, Theorem \ref{main} would directly follow if the map \eqref{eq silly map sick of this} were to extend continuously to a map
\begin{equation}\label{eq silly map sick of this2}
\CC(\Sigma)\to\CC^{\orb}(\RO).
\end{equation}
It is however easy to see that such an extension does not exist: for any three essential $\alpha,\beta,\gamma\in\pi_1(\Sigma)$ with $\beta\in\Ker(\pi_1(\Sigma)\to\pi_1^{\orb}(\RO))$ we have that $\frac 1{2n}[\alpha^n,\beta]\gamma$ converges when $n\to\infty$ in $\CC(\Sigma)$ to $\alpha$ but is mapped to $\frac 1{2n}\gamma\in\CC^{\orb}(\RO)$ which converges to 0. We by-pass this problem by choosing the representative $\eta_0$ of $\gamma_0$ so that the currents of the form $\frac 1L\eta$ with $\eta\in\Map(\Sigma)\cdot\eta_0$ are all contained in a closed subset of the set of currents on $\CC(\Sigma)$ to which the map \eqref{eq silly map sick of this} actually extends continuously. We choose $\eta_0$ to be {\em as simple as possible} in some precise sense given in Section \ref{sec:as simple as possible}. That the so chosen $\eta_0$ has the desired property follows from Proposition \ref{main proposition}, the technical result at the core of this paper. This proposition basically asserts that the images in $\tilde\RO$ of geodesics in $\tilde\Sigma$ which are as simple as possible are uniformly quasigeodesic. 

\subsection*{Non-orientable orbifolds}It is known that, at least as stated, the limit \eqref{eq-maryam} does not hold for non-orientable surfaces \cite{Gendulphe, Magee} and this is why we assumed in the theorems above that the orbifold is orientable. It is however worth noting that all results here remain true for non-orientable orbifolds whose underlying topological space is an orientable surface. An example is $D\Sigma/\tau$ where $D\Sigma$ is the double of $\Sigma$, an orientable surface with boundary, and $\tau$ is the involution interchanging the copies of the surface. The reason why the theorems remain true is that, up to passing to finite index subgroups, the mapping class group of such an orbifold is isomorphic to the mapping class group of an orientable surface for which we know that the analogue of Theorem \ref{main} holds. Anyways, we decided against extending the theorems above to this kind of non-orientable orbifolds since (1) it would make the paper much harder to read and (2) we do not have any concrete applications in mind.

\subsection*{Plan of the paper}
In Section \ref{sec:orbifolds} we recall some facts and definitions about orbifolds, maps between orbifolds, the mapping class group of orbifolds, and such. In Section \ref{sec:as simple as possible} we state precisely what we mean by {\em as simple as possible} and state, without proof, Proposition \ref{main proposition}. In Section \ref{sec:proofs of theorems} we recall a few facts about currents and, assuming Proposition \ref{main proposition}, prove Theorem \ref{main} and the other results mentioned above. In Section \ref{sec:lemmas} we prove a few facts needed in Section \ref{sec:proof of proposition}, where we prove Proposition \ref{main proposition}.

\subsection*{Acknowledgements} This has been one of those projects that for whatever reason take a long time to be completed. So long in fact that it is be impossible to make a comprehensive list of everyone we owe our gratitude to, and wishing not to be unfair we thank nobody---ingen n\"amnd ingen gl\"omd. With one exception, because the first author has not forgotten that during the start of the project she was supported by Pekka Pankka's Academy of Finland project \#297258 at the University of Helsinki. 

\section{Orbifolds}\label{sec:orbifolds}

In this section we recall a few basics about orbifolds such as definitions, (hyperbolic) orbifolds as orbit spaces, and mapping class groups. We also fix some notation that we will use throughout the paper. This is why we encourage also readers who already know all about orbifolds to at least skim over this section. 

\subsection{Orbifolds per se}
An orbifold $\RO$ is a space which is locally modeled on the quotient space of euclidean space by a finite group action. More precisely, an {\em orbifold chart} of a Hausdorff paracompact topological space $\RO$ is a tuple $(U,\hat U,\Gamma,\phi)$ where $U\subset\RO$ and $\hat U\subset\BR^n$ are open, where $\Gamma$ is a finite group acting on $\hat U$, and where $\phi:\hat U/\Gamma\to U$ is a homeomorphism. An {\em orbifold atlas} is a collection $\{(U_i,\hat U_i,\Gamma_i,\phi_i)\vert\ i\in I\}$ of orbifold charts such that $\{U_i\vert\ i\in I\}$ is an open cover of $\RO$ closed under intersections and such that whenever $U_i\subset U_j$ there are (1) a group homomorphism $f_{i,j}:\Gamma_i\to\Gamma_j$ and (2) an $f_{i,j}$-equivariant embedding 
$$\hat\phi_{i,j}:\hat U_i\to\hat U_j$$
such that the diagram
$$\xymatrix{\hat U_i\ar[d]\ar[r]^{\hat\phi_{i,j}} & \hat U_j\ar[d]\\ \hat U_i/\Gamma_i\ar[d]_{\phi_i}\ar[r] & \hat U_j/\Gamma_j \ar[d]^{\phi_j}\\ U_i \ar@{^{(}->}[r] & U_j}$$
commutes. An {\em orbifold} is then a Hausdorff paracompact space endowed with an orbifold atlas.

The orbifold is orientable if all group actions $\Gamma_i\actson\hat U_i$ and all embeddings $\hat\phi_{i,j}$ are orientation preserving. Similarly if we replace orientable by smooth. An orbifold with boundary is defined in the same way but this time the sets $\hat U_i$ are assumed to be open in $\BR_{\le 0}\times\BR^{n-1}$. An $n$-dimensional hyperbolic orbifold is one where the sets $\hat U_i$ are contained in $\BH^n$, where the actions $\Gamma_i\actson\hat U_i$ preserve the hyperbolic metric, and where the maps $\hat\phi_{i,j}$ are isometric embeddings. To define what is a hyperbolic orbifold with totally geodesic boundary then one copies what we just wrote, only replacing $\BH^n$ by a closed half-space therein.  

As is the case in the world of manifolds, orbifolds have maximal orbifold atlases, orientable orbifolds have maximal orientable orbifold atlases, smooth orientable orbifolds have maximal smooth orientable orbifold atlases, and so on. We will always assume that our orbifolds (with adjectives) are equiped with maximal atlases (with adjectives). 

We refer to \cite{Scott} for more on orbifolds.

\subsection{Singular points}
A point $p$ in an orbifold $\RO$ is {\em singular} if there are an orbifold chart $(U,\hat U,\Gamma,\phi)$ and $\hat p\in\hat U$ with $\phi(\hat p)=p$ and satisfying that $\Stab_\Gamma(\hat p)\neq \Id$. A point which is not singular is {\em regular}. We denote by $\sing(\RO)$ the set of singular points of $\RO$. 

The {\em singular set} $\sing(\RO)$ is a closed subset of $\RO$. It might be empty, but also its complement might be empty. It is actually sometimes really important to allow oneself to work with orbifolds with $\sing(\RO)=\RO$---not the simplest example one can find, but the moduli space $\CM_{2,0}$ of closed Riemann surfaces of genus $2$ is such an orbifold. However, 
\begin{quote}
all orbifolds in this paper are such that $\sing(\RO)$ is a proper subset of $\RO$.
\end{quote}
In the cases we are interested in, namely compact orbifolds which are orientable, connected and 2-dimensional we have that $\sing(\RO)$ is in fact a finite set of points in the interior of $\RO$. 

\begin{bem}
Whenever we need to choose a base point in our orbifold $\RO$, for example when working with the fundamental group $\pi_1^{\orb}(\RO)$, then we will assume without further mention that the base point is regular. The reader might amuse themselves by thinking about what the right notion of base point in the category of orbifolds would be if they allowed singular points to be base points.
\end{bem}

\subsection{Maps between orbifolds}\label{subsec:maps between orbifolds}
A map $f:\RO\to\RO'$ between two orbifolds is then a continuous map such that whenever $\phi_i:\hat U_i/\Gamma_i\to U_i$ and $\phi_i':\hat U_i'/\Gamma_i'\to U_i'$ are orbifolds charts for $\RO$ and $\RO'$ with $f(U_i)\subset U_i'$ then there are a homomorphism $\Gamma_i\to\Gamma_i'$ and an equivariant continuous map $\hat f_i:\hat U_i\to \hat U_i'$ such that the obvious diagram
$$\xymatrix{\hat U_i\ar[r]^{\hat f_i}\ar[d] & \hat U_i' \ar[d] \\ U_i\ar[r]_{f} & U_i'}$$
commutes. If $\RO$ and $\RO'$ are smooth orbifolds and if the $\hat f_i$ are smooth, then $f$ is said to be smooth. Orbifolds, and maps between orbifolds form a category. And the same for smooth orbifolds and smooth maps between them.

If $(U,\hat U,\Gamma,\phi)$ is an orbifold chart of an orbifold $\RO$, then $([0,1]\times U,[0,1]\times\hat U,\Gamma,\Id\times\phi)$ where $g\in\Gamma$ acts on $[0,1]\times \hat U$ via $g(t,x)=(t,gx)$ is an orbifold chart of $[0,1]\times\RO$. The collection of all so obtained orbifold charts forms an orbifold atlas, giving $[0,1]\times\RO$ the structure of an orbifold. It thus makes sense to say that two orbifold maps
$$f,f':\RO\to\RO'$$
are {\em homotopic in the category of orbifolds} if there is an orbifold map
$$F:[0,1]\times\RO\to\RO',\ \ F(t,x)=F_t(x)$$
with $F_0=f$ and $F_1=f'$. 

Anyways, armed with the notion of homotopy of orbifold maps one can define the {\em orbifold fundamental group} $\pi_1^{\orb}(\RO)$ of $\RO$ exactly as one does for the usual fundamental group, just replacing homotopies by orbifold homotopies. One should note that any two orbifold maps which are homotopic as orbifold maps are also homotopic as maps between topological spaces, but that the converse does not need to be true. In fact, there are plenty of orbifolds which are simply connected as topological spaces but whose orbifold fundamental group is non-trivial, meaning that there are orbifold maps $\gamma:\BS^1\to\RO$ which, as orbifold maps, are not homotopic to constant maps. As is the case for manifolds, in the category of orbifolds free homotopy classes of {\em curves} $\gamma:\BS^1\to\RO$  correspond to conjugacy classes in the orbifold fundamental group. We make the following convention:

\begin{quote}
If $\RO$ is a compact orbifold with boundary then we will say that a curve is {\em essential} if it is not freely homotopic into the boundary and if the associated free homotopy class is that of an infinite order element in the orbifold fundamental group. We will also denote by $\CS^{\orb}(\RO)$ the set of all free homotopy classes of essential curves in $\RO$.
\end{quote}

\subsection{Orbifolds as orbit spaces}
Following word-by-word the usual construction of the universal cover of a manifold but replacing homotopies by homotopies of orbifold maps one gets the orbifold universal cover $\tilde\RO$ of the orbifold $\RO$. As is the case for manifolds, the fundamental group $\pi^{\orb}(\RO)$ acts discretely on the universal cover $\tilde\RO$. Similarly, orbifold maps $f:\RO\to\RO'$ between orbifolds induce homomorphisms $f_*:\pi_1^{\orb}(\RO)\to\pi_1^{\orb}(\RO')$ between the associated orbifold fundamental groups and lift to $f_*$-equivariant maps $\tilde f:\tilde\RO\to\tilde\RO'$ between the universal covers. 

Orbifolds whose universal cover is a manifold are said to be {\em good}. And they deserve that name because working with them is much easier than working with general orbifolds. For example there is a pretty concrete description of the orbifold charts for good orbifolds $\RO$. They are namely of the form
$$\phi:\hat U/H\to U$$
where $H\subset\pi_1^{\orb}(\RO)$ is a finite subgroup, where $\hat U\subset\tilde\RO$ is an open connected subset with $H\hat U=\hat U$ and $g\hat U\cap\hat U=\emptyset$ whenever $g\notin H$, where $U$ is the image of $\hat U$ under the universal covering map $\pi:\tilde\RO\to\RO$, and where finally $\phi$ is the map given by $\phi(xH)=\pi(x)$.

Hyperbolic orbifolds, if one wants with geodesic boundary, are good. These are the orbifolds we will be interested in. We fix now the notation that we will be using from this point on:
\begin{quote}
{\bf Notation.} 
Let $\tilde\RO\subset\BH^2$ be a closed connected (2-dimensional) subset of the hyperbolic plane with possibly empty geodesic boundary, let $\Gamma\subset\PSL_2\BR$ be a discrete subgroup which preserves $\tilde\RO$ and such that the induced action $\Gamma\actson\tilde\RO$ is cocompact, and denote by
$$\RO=\tilde\RO/\Gamma$$
the associated hyperbolic orbifold. When needed, we will refer to the hyperbolic metric on both $\RO$ and $\tilde\RO$ by $\rho_{\hyp}$. Finally, we also write 
$$\sing(\Gamma)=\{p\in\tilde\RO\text{ with }\Stab_\Gamma(p)\neq\Id\}$$
for the set of points in $\tilde\RO$ with non-trivial stabilizer, that is the preimage of $\sing(\RO)$ under the map $\tilde{\RO}\to\RO$. 
\end{quote}
In this setting, $\tilde\RO$ is the orbifold universal cover of $\RO$ and $\Gamma=\pi_1^{\orb}(\RO)$ is its orbifold fundamental group. 
 
It is not hard to see that the orbifolds $\RO$ we are interested in are homeomorphic as topological spaces to surfaces, that is to 2-dimensional manifolds. Such homeomorphisms do however destroy the orbifold structure. In fact, much more information is encoded in the surface that we get by deleting the singular points of $\RO$. Since we want to work with compact surfaces, we instead delete small balls around the singular points.

\subsection{The surface associated to a hyperbolic orbifold $\RO$}\label{sec surface associated to orbifold}
Continuing with the same notation let $\RO=\tilde\RO/\Gamma$ be a compact orientable hyperbolic 2-orbifold with possibly empty totally geodesic boundary. We choose now two positive constants $\epsilon$ and $\delta$ which will accompany us throughout the paper. Other than being very small, say $\epsilon<10^{-10}$, here are the conditions that $\epsilon$ has to satisfy:
\begin{itemize}
\item[(C1)] $200\epsilon$ is less than the length of the shortest non-trivial periodic $\rho_{\hyp}$-geodesic in $\RO$,
\item[(C2)] $200\epsilon$ is less than the minimal distance between any two points in $\sing(\Gamma)$, and
\item[(C3)] $200\epsilon$ is less than the distance between any point in $\D\tilde\RO$ and any point in $\sing(\Gamma)$.
\end{itemize} 
When it comes to $\delta$ we will later give a fourth condition (see (C4) in Section \ref{sec:lemmas}) that it has to satisfy but for now we just assume that $3\delta<\epsilon$. Note that this implies that the $\delta$-balls around points in $\sing(\Gamma)$ are disjoint of each other and do not meet $\D\tilde\RO$. This means that
\begin{equation}\label{eq associated surface}
\hat\Sigma=\tilde\RO\setminus\CN_{\hyp}(\sing(\Gamma),\delta)
\end{equation}
is a smooth surface with boundary, where
$$\CN_{\hyp}(X,r)=\{p\in\tilde\RO\text{ with }d_{\hyp}(p,X)< r\}.$$

Note that the action of $\Gamma=\pi_1^{\orb}(\RO)$ on $\tilde\RO$ induces an action on $\hat\Sigma$ which is not only discrete but also free. We refer to the quotient surface
$$\Sigma=\hat\Sigma/\Gamma$$
as {\em the surface associated to the orbifold $\RO$} and denote its universal cover by $\tilde\Sigma$. By construction, it is also the universal cover of $\hat\Sigma$. In fact, $\hat{\Sigma}$ is the cover of $\Sigma$ corresponding to the normal subgroup of $\pi_1(\Sigma)$ generated by all loops homotopic into $\partial\Sigma\setminus\D\RO$.

\begin{bem} 
We denote by $B_{\hyp}(q,r)\subset\tilde{\RO}$ the hyperbolic ball of radius $r$ around $q$. Equivalently,
$$B_{\hyp}(q,r)=\CN_{\hyp}(\{q\},r).$$ 
Also, abusing terminology we will not distinguish between $\CN_{\hyp}(X,r)$ or $B_{\hyp}(q,r)$ and their closures. Thats is, both open balls and closed balls, and open neighborhoods and closed neighborhoods are denoted using the same symbol.
\end{bem}

\subsection{Mapping class groups of the orbifold and of the associated surface}\label{sec orbifold mapping class group}
As is the case for manifolds, one can say anything one wants to say about the orbifold $\RO=\tilde\RO/\Gamma$ in terms of $\Gamma$-equivariant objects in the universal cover $\tilde\RO$. For example, the group $\Homeo^{\orb}(\RO)$ of orbifold self-homeomorphisms of $\RO$ can be identified with 
$$\Homeo^{\orb}(\RO)=\Homeo_\Gamma(\tilde{\RO})/\Gamma$$
where 
$$\Homeo_\Gamma(\tilde\RO)=\{\tilde f\in\Homeo(\tilde\RO)\text{ with }\tilde f\Gamma\tilde f^{-1}=\Gamma\}$$
is the group of (topological) homeomorphisms of $\tilde\RO$ conjugating $\Gamma$ to itself. The mapping class group, in the category of orbifolds, of $\RO$ is then the group 
$$\Map^{\orb}(\RO)=\Homeo^{\orb}(\RO)/\Homeo_0^{\orb}(\RO)$$
where $\Homeo_0^{\orb}(\RO)$ is the identity component of $\Homeo^{\orb}(\RO)$.

The group $\Homeo_\Gamma(\tilde\RO)$ acts on the set
$$\sing(\Gamma)=\{p\in\tilde\RO\text{ with }\Stab_\Gamma(p)\neq\Id\}$$
of points with non-trivial stabilizer. It also acts on the set $\pi_0(\D\tilde\RO)$ of boundary component of $\tilde\RO$. It follows that the mapping class group acts on the finite sets $\sing(\Gamma)/\Gamma$ and $\pi_0(\D\tilde\RO)/\Gamma$.  The {\em pure mapping class group} 
$$\PMap^{\orb}(\RO)=\{\phi\in\Map^{\orb}(\RO)\text{ pointwise fixing }\sing(\Gamma)/\Gamma\text{ and }\pi_0(\D\tilde\RO)/\Gamma\}$$
is the finite index subgroup of $\Map^{\orb}(\RO)$ consisting of mapping classes which act trivially on these two sets.

Note now that the canonical inclusion $\Sigma\hookrightarrow\RO$ into our orbifold of the associated surface is an embedding in the category of orbifolds. We have however also other interesting maps $\Sigma\to\RO$, namely those which are the identity outside of a small neighborhood of the boundary of $\Sigma$ and which map $\Sigma\setminus\D\Sigma$ homeomorphically to $\RO\setminus\sing(\RO)$. Such maps are not homeomorphisms but they induce homorphisms between the group of homeomorphisms of $\Sigma$ acting trivially on $\pi_0(\D\Sigma)$ and the group of orbifold homeomorphisms of $\RO$. Any such map induces an isomorphism between the pure mapping class groups 
\begin{equation}\label{eq isomorphisms pure mapping class groups}
\PMap(\Sigma)\simeq\PMap^{\orb}(\RO)
\end{equation}
of $\Sigma$ and $\RO$, where 
$$\PMap(\Sigma)=\{\phi\in\Homeo(\Sigma)\text{ acting trivially on }\pi_0(\D\Sigma)\}/\Homeo_0(\Sigma).$$
It is well-known that every mapping class in $\Map(\Sigma)$ can be represented by a diffeomorphism. 

Although our definition of the mapping class group differs from theirs (we do not have twists around the boundary) we refer to the book \cite{Farb-Margalit} by Farb and Margalit for background on the mapping class group. 

\subsection{A metric on the associated surface $\hat\Sigma$}\label{sec:metric}
Although (locally) negatively curved from the point of view of comparison geometry, the restriction of the hyperbolic metric $\rho_{\hyp}$ to $\hat\Sigma$ is not as nice as one would wish. The problem is that, since the new boundary components are concave, geodesics are not uniquely determined by their tangent vectors at a point. In particular, distinct geodesics do not need to be transversal to each other. This is why we from now on endow $\hat\Sigma$ with a smooth Riemannian metric $\rho$ with the following properties: 
\begin{itemize}
\item $\rho$ is negatively curved and $\Gamma$-invariant.
\item The boundary of $\hat\Sigma$ is totally geodesic with respect to $\rho$.
\item Both $\rho$ and $\rho_{\hyp}$ agree on the subset $\tilde\RO\setminus\CN_{\hyp}(\sing(\Gamma),2\delta)$ of $\hat\Sigma$.
\item If $I\subset\tilde\RO$ is a $\rho_{\hyp}$-geodesic segment starting at a point $\sing(\Gamma)$ and with $\rho_{\hyp}$-length $3\delta$, then $I\cap\hat\Sigma$ is a $\rho$-geodesic segment perpendicular to the boundary of $\Sigma$.
\end{itemize} 
The reason why we impose this final condition is that, if $p\in\sing(\Gamma)$ and $r>2\delta$ are such that $\rho$ and $\rho_{\hyp}$ agree on $B_{\hyp}(p,r)\setminus B_{\hyp}(p,2\delta)$ then the $\rho_{\hyp}$ radial foliation $\CF$ of $B_{\hyp}(p,r)\setminus B_{\hyp}(p,\delta)$ is $\rho$-geodesic. It follows in particular that the restriction of the radial projection
$$B_{\hyp}(p,r)\setminus B_{\hyp}(p,\delta)\to\D B_{\hyp}(p,r)$$
to any $\rho$-geodesic segment $\eta$ which is not contained in a leaf of $\CF$ is monotonic in the sense that its derivative is never $0$. We thus get that simple $\rho$-geodesic segments $\eta\subset B_{\hyp}(p,r)\setminus B_{\hyp}(p,\delta)$ whose endpoints are in $\D B_{\hyp}(p,r)$ and meet each leaf of $\CF$ at most once---see Figure \ref{fig radial foliation}. We record this fact for later use:

\begin{lemma}\label{lem:meet rays once}
Suppose that $p\in\sing(\Gamma)$ and $r>2\delta$ are such that $\rho$ and $\rho_{\hyp}$ agree on $B_{\hyp}(p,r)\setminus B_{\hyp}(p,2\delta)$, and let $\eta\subset B_{\hyp}(p,r)\setminus B_{\hyp}(p,\delta)$ be a $\rho$-geodesic segment whose boundary points are contained in $\D B_{\hyp}(p,r)$. If $\eta$ is simple, then $\eta$ meets every $\rho_{\hyp}$-geodesic ray emanating out of $p$ at most once. In particular, $\eta$ has at most length $2\pi\sinh(r)$.\qed
\end{lemma}

\begin{figure}[h]
\includegraphics[width=0.4\textwidth]{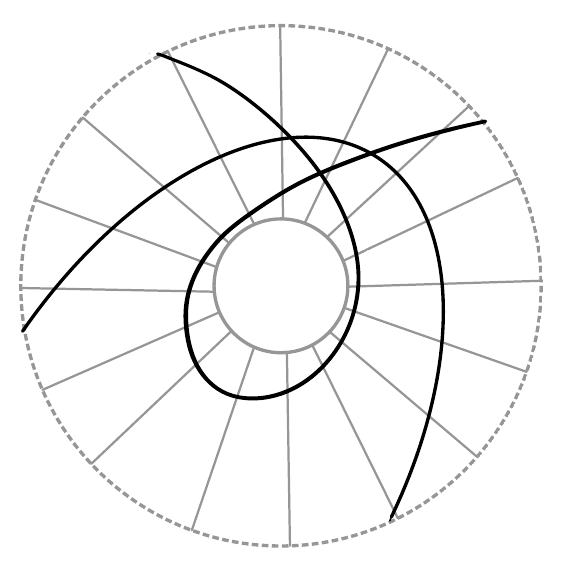}
\caption{Schematic representation of two $\rho$-geodesic segments in $B_{\hyp}(p,3\delta)\setminus B_{\hyp}(p,\delta)$.}
\label{fig radial foliation}
\end{figure}

We should comment on the existence of $\rho$. In fact, it is not hard to construct such a metric. For example, when working in standard hyperbolic polar coordinates $(r,\theta)$ in the ball $B_{\hyp}(p,3\delta)$ around $p\in\sing(\Gamma)$ one can take any 
\begin{equation}\label{go buy batteries}
\rho=dr^2+\phi(r)^2\cdot d\theta^2 
\end{equation}
where $\phi:[\delta,3\delta)\to(0,\infty)$ is a smooth function satisfying
$$ \phi''(\cdot)>0,\ \phi'(\delta)=0, \text{ and }\phi(s)=\sinh(s)\text{ for }s>2\delta.$$
The first condition on $\phi$ ensures that the sectional curvature $\kappa= \frac{-\phi''}{\phi}$ is negative, the second that $\{d_{\hyp}(p,\cdot)=\delta\}$ is totally geodesic, and the third that $\rho$ agrees with $\rho_{\hyp}$ on $B_{\hyp}(p,3\delta)\setminus B_{\hyp}(p, 2\delta)$. In particular, if we use the same function $\phi$ on each $3\delta$-ball around points in $\sing(\Gamma)$ and we set $\rho=\rho_{\hyp}$ outside those balls, then we obtain a $\Gamma$-invariant metric on the whole of $\hat{\Sigma}$. Finally note that the curves $t\mapsto (t, \theta)$, that is the $\rho_{\hyp}$-geodesic segments starting at $p$, are $\rho$-geodesic segments for any choice of $\phi$. In other words, also the fourth property we wanted our metric to satisfy holds. 

Note that $\Gamma$-invariance of $\rho$ implies that it descends to a metric on $\Sigma$ which we once again call $\rho$. Similarly, we denote also by $\rho$ the induced metric on the universal cover $\tilde{\Sigma}$.

\section{As simple as possible representatives}\label{sec:as simple as possible}
Continuing with the same notation let $\RO=\tilde\RO/\Gamma$ be a compact orientable hyperbolic orbifold and let $\Sigma=\hat\Sigma/\Gamma$ with $\hat\Sigma$ as in \eqref{eq associated surface} be the associated surface, endowed with the metric $\rho$ we just fixed. By construction, $\hat\Sigma$ is a connected subset of the universal cover $\tilde\RO$ of $\RO$. It follows that the inclusion $\Sigma\hookrightarrow\RO$ induces a surjective homomorphism
$$\pi_1(\Sigma)\to\pi_1^{\orb}(\RO)=\Gamma.$$
This means that every homotopically essential curve in $\RO$ is freely homotopic (in the category of orbifolds) to one contained in $\Sigma$. In this section we describe how to pick for curves in $\RO$ representatives in $\Sigma$ which are as simple as possible.

\begin{defi*}
A $\rho$-geodesic $\alpha:\BR\to \hat{\Sigma}$ whose image is not contained in $\D\hat\Sigma$ is {\em as simple as possible} if 
\begin{enumerate}
 \item it is injective, and
 \item for all $g\in\Gamma$ the geodesics $\alpha$ and $g(\alpha)$ are either identical or meet at most once.
\end{enumerate}
We say that a $\rho$-geodesic in $\Sigma$ is as simple as possible if its lifts to $\hat{\Sigma}$ are as simple as possible. Similarly, a $\rho$-geodesic in the universal cover $\tilde\Sigma$ of $\Sigma$ is as simple as possible if its images in $\hat\Sigma$ are as simple as possible. Finally, a homotopy class in $\Sigma$ is as simple as possible if its $\rho$-geodesic representative is as simple as possible.  
\end{defi*}

\begin{figure}[h]
\includegraphics[width=0.4\textwidth]{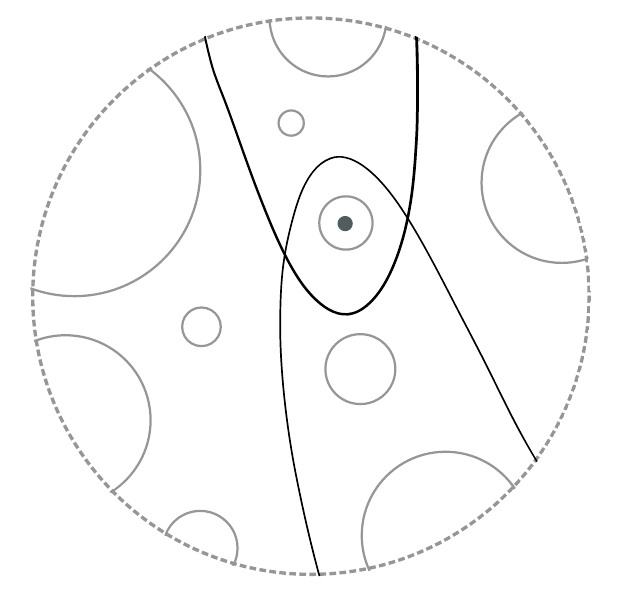}
\caption{Two geodesics (in black) in $\hat{\Sigma}$ which differ by an element in $\Gamma$ fixing the marked point: they are not as simple as possible.}
\label{fig almost simple}
\end{figure}

Before going any further we note that non-trivial closed $\rho_{\hyp}$-geodesics $\gamma:\BS^1\to\RO$ have representatives $\eta:\BS^1\to\Sigma\subset\RO$ which are as simple as possible. It suffices to choose $\eta\subset \Sigma$ to be a shortest representative of $\gamma$. Indeed, the fact that $\eta$ is shortest implies that its lifts to $\hat{\Sigma}$ have no bigons, showing that $\eta$ is as simple as possible. We record this fact for later use:

\begin{lemma}\label{lem banana}
Every $\rho_{\hyp}$-geodesic $\gamma:\BS^1\to\RO$ is freely homotopic, in the category of orbifolds, to a $\rho$-geodesic $\eta:\BS^1\to\Sigma\subset\RO$ which is as simple as possible.\qed
\end{lemma}

The reader might be wondering why instead of simply speaking of shortest representatives we choose something as clumsy as ``as simple as possible''. The reason is that the latter property is mapping class group invariant:

\begin{lemma}\label{lem as simple as possible mapping class group}
If $\eta\subset\Sigma$ is as simple as possible then $\phi(\eta)$ is also as simple as possible for every $\phi\in\PMap(\Sigma)$.
\end{lemma}

\begin{proof}
Abusing notation, denote the $\rho$-geodesic freely homotopic to $\eta$ by the same letter. To determine $\phi(\eta)$ choose first a representative $\varphi\in\Diff(\Sigma)$ of the mapping class $\phi$ and let $\varphi_*:\pi_1(\Sigma)\to\pi_1(\Sigma)$ be the homomorphism induced by $\varphi$---we can always choose $\varphi$ so that it fixes some point and take that point as the base point for the fundamental group. Note that $\varphi_*$ preserves the normal subgroup of $\pi_1(\Sigma)$ generated by loops freely homotopic into $\D\Sigma\setminus\D\RO$ and that $\pi_1(\hat\Sigma)\subset\pi_1(\Sigma)$ is nothing other than this subgroup. We get that $\varphi$ lifts to $\hat\Sigma$, or more precisely, that there is a $\varphi_*$-equivariant lift $\hat\varphi\in\Diff(\hat\Sigma)$. Now, if $\hat\eta$ is a lift of $\eta$ to $\hat\Sigma$ then we have for all $g\in\Gamma$ that
$$\hat\varphi(\hat\eta)\cap\varphi_*(g)\left(\hat\varphi(\hat\eta)\right)=\hat\varphi(\hat\eta)\cap\hat\varphi(g\hat\eta)=\hat\varphi(\eta\cap g\hat\eta).$$
Since $\eta$ was as simple as possible we get that $\hat\varphi(\hat\eta)$ is simple and that $\hat\varphi(\hat\eta)\cap\varphi_*(g)\hat\varphi(\hat\eta)$ intersects its individual $\Gamma$-translates at most once, and that if these intersections take place then they are transversal to each other. 

It follows that the image of $\varphi(\eta)$ in $\Sigma$ has no bigons. The same is true for $(\phi(\eta))_*$, the geodesic in $(\Sigma, \rho)$ freely homotopic to $\varphi(\eta)$. Now, \cite[Theorem 2.1]{Hass-Scott} implies that these two curves are not only freely homotopic to each other but also transversely freely homotopic to each other. This means in particular that intersection points are neither destroyed nor created during the homotopy. Hence, each lift of $(\phi(\eta))_*$ to $\hat\Sigma$ meets its individual $\Gamma$-translates in at most one point. In other words, $\phi(\eta)$ is as simple as possible.
\end{proof}

The reason why we are interested in $\rho$-geodesics in $\hat\Sigma$ which are as simple as possible is that, as we will see shortly, this topological property implies that they are uniform quasigeodesics with respect to the hyperbolic metric. Recall that a continuous curve $\alpha:\BR\to\tilde\RO$ is {\em $A$-quasigeodesic} if we have
$$A\cdot\vert t-s\vert+A\ge d_{\hyp}(\alpha(s),\alpha(t))\ge\frac 1A\vert s-t\vert-A$$
for all $s,t\in\mathbb{R}$. It is quasigeodesic if it is $A$-quasigeodesic for some $A\ge 1$.

We are now ready to state the key technical result of this paper:

\begin{prop}\label{main proposition}
Let $\RO$ be as in the statement of Theorem \ref{main}, $\hat{\Sigma}$ as in \eqref{eq associated surface}, and $\rho$ the metric on $\hat{\Sigma}$ constructed in Section \ref{sec:metric}.  There exists $A\ge 1$ such that any unit speed $\rho$-geodesic $\alpha: \BR\to\hat\Sigma$ which is (1) a quasigeodesic in $(\tilde{\RO},\rho_{\hyp})$ and (2) as simple as possible, is actually $A$-quasigeodesic in $(\tilde{\RO},\rho_{\hyp})$. 
\end{prop}

Proposition \ref{main proposition} will be proved in Section \ref{sec:proof of proposition}. We just add now a few comments:
\medskip

\noindent{\bf (1)} Note that in Proposition \ref{main proposition} we cannot simply drop the assumption that $\alpha$ is a quasigeodesic in $(\tilde{\RO},\rho)$. For example, if $\alpha$ is a lift of a simple geodesic in $\Sigma$ which spirals in both directions onto components of $\partial\Sigma\setminus\partial\RO$ then it is as simple as possible but not a quasigeodesic and in particular not an $A$-quasigeodesic for any choice of $A$. However, this is basically the only case we have to rule out because we could replace the condition that $\alpha$ is a quasigeodesic in the proposition by the assumption that $\alpha$ does not accumulate on a compact component of $\partial{\hat{\Sigma}}$ in either direction. We leave it however as it is because the curves we will be interested in are automatically quasigeodesics: they are lifts $\hat\eta$ to $\hat\Sigma$ of representatives $\eta$ in $\Sigma$ of essential curves $\gamma$ in $\RO$. 

\noindent{\bf (2)} Suppose that $\gamma\in\CS^{\orb}(\RO)$ is an essential curve in $\RO$. Proposition \ref{main proposition} implies that we have  
$$\ell_{\rho}(\eta)\le A\cdot\ell_{\hyp}(\gamma)$$
for any $\rho$-geodesic representative $\eta\subset\Sigma$ of $\gamma$ which is as simple as possible. 
It follows that $\gamma$ only has finitely many such representatives.

\noindent{\bf (3)} 
On the other hand, as simple as possible representatives are not unique. In fact, a primitive closed geodesic in $\RO$ which goes through $k$ cone points of odd order and none of even order, has at least $2^k$ representatives in $\Sigma$ which are as simple as possible: at each one of those $k$ points, steer slightly either right or left to avoid hitting the cone point. And this is not optimal because perturbing the metric slightly one can get the geodesic off $\sing(\RO)$, and representatives that were as simple as possible stay as simple as possible. 

\noindent{\bf (4)} 
The construction sketched in (3) shows that ``shortest" and ``as simple as possible" are not the same thing.

\section{Main results}\label{sec:proofs of theorems}
In this section we prove Theorem \ref{main} assuming Proposition \ref{main proposition}. However, before doing so we have  to recall a few facts about currents and about Mirzakhani's counting theorem.

\subsection{Currents}\label{sec currents}
Let $X$ be a simply connected negatively curved surface with possibly empty totally geodesic boundary, and let $G\subset\Isom_+(X)$ be a discrete subgroup of orientation preserving isometries with $X/G$ compact. We will be interested in the following two possible cases:
\begin{itemize}
\item $X=\tilde\RO\subset\BH^2$ is the universal cover of our compact hyperbolic orbifold $\RO=\tilde\RO/\Gamma$ and $G=\pi_1^{\orb}(\RO)=\Gamma$ is its fundamental group. 
\item $X=(\tilde\Sigma,\rho)$ is the universal cover of $\Sigma$ endowed with the metric $\rho$ and $G=\pi_1(\Sigma)$ is its fundamental group.
\end{itemize}
Let $\CG(X)$ be the set of all unoriented bi-infinite geodesics in $X$. The action $G\actson X$ induces an action of $G$ on $\CG(X)$. A current on $X/G$ is a $G$-invariant Radon measure on $\CG(X)$. Let $\CC^{\orb}(X/G)$ be the space of all currents on $X/G$ endowed with the weak-*-topology. It is a Hausdorff, metrizable, second countable, and locally compact space, and the projectivized space $\BP\CC^{\orb}(X/G)$ is compact. 

\begin{bem}
We insist that $X/G$, and thus our orbifold $\RO$, is compact because this is what guarantees that the space $\CC^{\orb}(X/G)$ is locally compact. 
\end{bem}

There are plenty of currents. In fact there is a natural homeomorphism between $\CC^{\orb}(X/G)$ and the space of geodesic flow invariant Radon measures on the projectivized unit tangent bundle $PT^1X/G$ supported by the set of bi-infinite orbits. For example, every primitive closed unit speed geodesic $\gamma$ in $\CG(X)$, or equivalently every unoriented periodic orbit of the geodesic flow, yields a geodesic flow invariant measure on $PT^1X/G$: the measure of $U$ is the arc length of $\gamma\cap U$. The current associated to this measure is called the {\em counting current associated to the geodesic $\gamma$.} The counting current determines the original geodesic $\gamma$---this justifies referring to the current and the geodesic by the same letter---and the name is explained because, for a set of geodesics $V\subset\CG(X)$ the value of $\gamma(V)$ is nothing other than the number of lifts of $\gamma$ to $X$ which belong to $V$.

Note that every essential curve $\gamma$ in $X/G$ (in the sense that we gave to the word essential at the end of Section \ref{subsec:maps between orbifolds}) is freely homotopic to a unique geodesic $\gamma_*$ in $X/G$. In this case we denote the associated counting current by $\gamma$ instead of $\gamma_*$. We hope that this will not cause any confusion.

\begin{bem}
If the action of $G$ on $X$ is free then we drop the superscript ``$\orb$". For example we write $\CC(\Sigma)$ instead of $\CC^{\orb}(\Sigma)$. We use this superscript to avoid mixing up currents for the orbifold $X/G$ and currents for the underlying topological surface.
\end{bem}

Currents were introduced by Bonahon and we refer to his papers \cite{BonahonFrench,Bonahon2,Bonahon} for details and background. See also \cite{AL}. However, although all these sources are highly recommended, the reader will not be surprised on hearing that we will mostly follow the same notation and terminology as in our book \cite{Book}.

\subsection{Mirzakhani's counting theorem}\label{need batteries}
As we mentioned already in the introduction, Theorem \ref{main} is well-known in the case that we are working with surfaces instead of orbifolds. In that case we have the following result \cite[Theorem 8.1]{Book}:

\begin{named}{Mirzakhani's counting theorem}
Let $\Sigma$ be a compact connected orientable surface of genus $g$ and with $r$ boundary components and suppose that $3g-3+r>0$. Let also $\eta_0\subset\Sigma$ be a homotopically primitive essential curve. Then there are constants $\FC^{\PMap}(\eta_0),\FB^{\PMap}_{g,r}>0$ such that
$$\lim_{L\to\infty}\frac 1{L^{6g-6+2r}}\sum_{\eta\in\PMap(\Sigma)\cdot\eta_0}\delta_{\frac 1L\eta}=\frac{\FC^{\PMap}(\eta_0)}{\FB^{\PMap}_{g,r}}\cdot\FM^\Sigma_{\Thu}$$
Here $\delta_{\frac 1L\eta}$ is the Dirac measure on $\CC(\Sigma)$ centered at the current $\frac 1L\eta$, $\FM_{\Thu}^\Sigma$ is the Thurston measure on $\CC(\Sigma)$, and the convergence takes place with respect to the weak-*-topology on the space of Radon measures on $\CC(\Sigma)$.
\end{named}

\begin{bem}
The counting theorem remains true if we replace the pure mapping class group by any other finite index subgroup of the mapping class group. However the obtained multiple of the Thurston measure depends on the subgroup in question. This explains the superscript $\PMap$ in the constants $\FC^{\PMap}(\eta_0)$ and $\FB^{\PMap}_{g,r}$. See \cite[Exercise 8.2]{Book} for explicit formulas for the dependence of the constants on the chosen subgroup of the mapping class group. 
\end{bem}

Since we named the above theorem after Maryam Mirzakhani while referring to our book \cite{Book} we should add a brief comment on the genesis of this theorem. For simple curves, Mirzakhani proved this theorem in \cite{Maryam0,Maryam1} but for general curves the history is slightly more complicated. To explain why, note that the theorem implies that
$$\lim_{L\to\infty}\frac{\{\gamma\in\PMap(\Sigma)\text{ with }F(\gamma)\le L\}}{L^{6g-6+2r}}=\frac{\FC^{\PMap}(\gamma_0)}{\FB^{\PMap}_{g,r}}\cdot\FM_{\Thu}^\Sigma(\{F(\cdot)\le 1\})$$
whenever $F:\CC(\Sigma)\to\BR_+$ is a continuous, positive and homogenous function (compare with \cite[Theorem 9.1]{Book} or with the proof of Theorem \ref{sat3} below). In \cite{Maryam2} Maryam proved the existence of the latter limit in the case that $F$ is the hyperbolic length. At the same time, we were also investigating the same problem and we proved in \cite{ES} that every sublimit of the sequence in the counting theorem has a subsequence which converges to a multiple of the Thurston measure. The existence of the limit in the counting theorem follows if one combines these two facts \cite{Maryam2,ES}. This was clear to both Mirzakhani and ourselves at the time. A problem with that state of affairs was that Mirzakhani's arguments and ours come from different places, and this made things a bit too opaque. For example there was some confusion about the chosen normalization for the Thurston measure. This meant that it was not obvious how the arising constants should be understood (this problem was, to some extent, solved in \cite{Monin, RS}). Finally, or maybe finally for the time being, a unified proof of the counting theorem as stated above was provided in \cite{Book}. The constants $\FC^{\PMap}(\gamma_0)$ and $\FB^{\PMap}_{g,r}$ in the statement of the counting theorem above are as given in \cite[Chapter 8]{Book}, and the Thurston measure is defined to be the scaling limit
\begin{equation}\label{eq defintion thurston measure}
\FM^\Sigma_{\Thu}=\lim_{L\to\infty}\sum_{\gamma\in\CM\CL_\BZ}\delta_{\frac 1L\gamma}.
\end{equation}
Anyways, let us return to the concrete topic of this paper.

\subsection{Proof of Theorem \ref{main}}
We prove now our main theorem assuming Proposition \ref{main proposition}. As mentioned earlier, the proposition will be proved in Section \ref{sec:proof of proposition}.

\begin{named}{Theorem \ref{main}}
Let $\RO$ be a compact orientable non-exceptional hyperbolic orbifold with possibly empty totally geodesic boundary and let $\CC^{\orb}(\RO)$ be the associated space of geodesic currents. There is a Radon measure $\FM_{\Thu}$ on $\CC^{\orb}(\RO)$ such that for any $\gamma_0\in\CS^{\orb}(\RO)$ we have
$$\lim_{L\to\infty}\frac 1{L^{6g-6+2r}}\sum_{\gamma\in\Map^{\orb}(\RO)\cdot\gamma_0}\delta_{\frac 1L\gamma}=C(\gamma_0)\cdot\FM_{\Thu}$$
for some positive constant $C(\gamma_0)>0$. Here $g$ is the genus of the orbifold $\RO$, $r$ is the sum of the numbers of singular points and boundary components, and $\RO$ is non-exceptional if $(g,r)\neq(0,3)$. Moreover $\delta_{\frac{1}{L}\gamma}$ stands for the Dirac measure on $\CC^{\orb}(\RO)$ centered at $\frac 1L\gamma$, and the convergence takes place with respect to the weak-*-topology on the space of Radon measures on $\CC^{\orb}(\RO)$.
\end{named}

As we already mentioned in the introduction, the idea of the proof is to show that our given homotopy class $\gamma_0$ has a representative $\eta_0$ in $\Sigma$ such that the measures 
$$\frac 1{L^{6g-6+2r}}\sum_{\eta\in\PMap(\Sigma)\cdot\eta_0}\delta_{\frac 1L\eta}$$
inside the limit in Mirzakhani's counting theorem are supported by a closed subset of $\CC(\Sigma)$ which maps continuously to $\CC^{\orb}(\RO)$. In fact we will choose $\eta_0$ to be a representative of $\gamma_0$ which is as simple as possible.

Anyways, with $A$ as in Proposition \ref{main proposition} let
$$\CQ(A)\subset\CG(\Sigma)$$
be the set of unit speed geodesics $\alpha$ in $(\tilde\Sigma,\rho)$ with the property that the composition of the maps
$$\xymatrix{\BR\ar[r]^\alpha & (\tilde\Sigma,\rho)\ar[r] &(\hat\Sigma,\rho) \ar@{^{(}->}[r] & (\tilde\RO,\rho_{\hyp})}$$
is an $A$-quasigeodesic. In these terms, Proposition \ref{main proposition} asserts that if $\alpha:\BR\to(\tilde\Sigma,\rho)$ is a unit speed geodesic whose image in $(\hat\Sigma,\rho)$ is as simple as possible, then $\alpha\in\CQ(A)$. Recalling now that by Lemma \ref{lem banana} the homotopy class of every closed primitive and essential geodesic $\gamma_0$ in $\RO$ can be represented by a $\rho$-geodesic $\eta_0:\BS^1\to(\Sigma,\rho)$ which is as simple as possible, and that from Lemma \ref{lem as simple as possible mapping class group} we get that the property of being as simple as possible is mapping class group invariant, then we get the following fact that we state as a lemma for later reference:

\begin{lemma}\label{lem well chosen lift}
If $\eta_0:\BS^1\to(\Sigma,\rho)$ is any essential $\rho$-geodesic which when considered as a map into $\RO$ is as simple as possible, then the measure
$$\sum_{\eta\in\PMap(\Sigma)\cdot\eta_0}\delta_{\frac 1L}\eta$$
is supported by $\CQ(A)$ for all $L>0$.\qed
\end{lemma}

Now, the fact that the quasigeodesic constant $A$ is fixed implies that $\CQ(A)$ is a closed subset of $\CG(\Sigma)$. Recall also that every $A$-quasigeodesic in $\tilde\RO$, and in particular the image under $\tilde\Sigma\to\hat\Sigma\hookrightarrow\tilde\RO$ of each element of $\CQ(A)$, is at bounded distance of a $\rho_{\hyp}$-geodesic where the bound just depends on $A$. In this way we get a continuous map
\begin{equation}\label{map L geodesics}
\CQ(A)\to\CG(\tilde\RO)
\end{equation}
equivariant under the homomorphism $\pi_1(\Sigma)\to\pi_1^{\orb}(\RO)=\Gamma$. Now, pushing currents forward with \eqref{map L geodesics} (at the end of the day currents are measures) we get a continuous map
\begin{equation}\label{map L geodesics currents}
\Pi:\{\lambda\in\CC(\Sigma)\text{ supported by }\CQ(A)\}\to\CC^{\orb}(\RO).
\end{equation}
from the closed subset of $\CC(\Sigma)$ consisting of currents supported by the closed set $\CQ(A)$ to the space of currents on $\RO$. 

The map $\Pi$ given in \eqref{map L geodesics currents} induces in turn a continuous map
\begin{equation}\label{map L geodesics currents measures}
\Pi_*:\text{measures  on }
\left\{
\begin{array}{c}
\lambda\in\CC(\Sigma)\\ \text{supported by }\CQ(A)
\end{array}
\right\}
\to\text{measures on }\CC^{\orb}(\RO)
\end{equation}
from the space of Radon measures on $\{\lambda\in\CC(\Sigma)\text{ supported by }\CQ(A)\}$ to the space of Radon measures on $\CC^{\orb}(\RO)$. 

For $\gamma_0$ as in the statement of Theorem \ref{main}, let $\eta_0$ be as provided by Lemma \ref{lem banana}, and note that we get from Lemma \ref{lem well chosen lift} that the measure $\sum_{\eta\in\Map(\Sigma)\eta_0}\delta_{\frac 1L}\eta$ is supported by the domain of \eqref{map L geodesics currents}. Its image under \eqref{map L geodesics currents measures} is in fact nothing other than the measure $\sum_{\gamma\in\Map^{\orb}(\RO)\cdot\gamma_0}\delta_{\frac 1L}\gamma$. Applying $\Pi_*$ to both sides of the limit in Mirzakhani's counting theorem we get:
\begin{align}\label{viv has a zoom meeting}
\nonumber \Pi_*\left(\frac{\FC^{\PMap}(\eta_0)}{\FB^{\PMap}_{g,r}}\cdot\FM^\Sigma_{\Thu}\right)
&=\Pi_*\left(\lim_{L\to\infty}\frac 1{L^{6g-6+2r}}\sum_{\eta\in\PMap(\Sigma)\cdot\eta_0}\delta_{\frac 1L\eta}\right)\\
&=\lim_{L\to\infty}\frac 1{L^{6g-6+2r}}\Pi_*\left(\sum_{\eta\in\PMap(\Sigma)\cdot\eta_0}\delta_{\frac 1L\eta}\right)\\
\nonumber&=\lim_{L\to\infty}\frac 1{L^{6g-6+2r}}\sum_{\gamma\in\PMap^{\orb}(\RO)\cdot\gamma_0}\delta_{\frac 1L\gamma}
\end{align}
Now the $\Map^{\orb}(\RO)$-orbit of $\gamma_0$ is a disjoint union of orbits under $\PMap^{\orb}(\RO)$, more precisely of 
$\frac{|\Map^{\orb}(\RO)/\PMap^{\orb}(\RO)|}{|\Stab_{\Map^{\orb}(\RO)}(\gamma_0)/\Stab_{\PMap^{\orb}(\RO)}(\gamma_0)\vert}$ orbits. Theorem \ref{main} follows when we apply \eqref{viv has a zoom meeting} to each one of these orbits and we set 

$$C(\gamma_0)=\frac{|\Map^{\orb}(\RO)/\PMap^{\orb}(\RO)|}{|\Stab_{\Map^{\orb}(\RO)}(\gamma_0)/\Stab_{\PMap^{\orb}(\RO)}(\gamma_0)\vert}\cdot \frac{\FC^{\PMap}(\eta_0)}{\FB^{\PMap}_{g,r}}$$
and
$$\FM_{\Thu}=\Pi_*\left(\FM^\Sigma_{\Thu}\right).$$
We have proved Theorem \ref{main}.\qed
\medskip

Rather, we have proved Theorem \ref{main} while assuming Proposition \ref{main proposition}. Anyways, before proving the proposition let us prove the other theorems mentioned in the introduction and comment briefly on the measure $\FM_{\Thu}$.

\subsection{A comment on the measure $\FM_{\Thu}$}\label{sec thurston measure}
In the course of the proof of Theorem \ref{main} we identified the measure $\FM_{\Thu}$ as the push forward of the Thurston measure $\FM_{\Thu}^\Sigma$ associated to $\Sigma$ under the map $\Pi_*$. We give now a slightly more intrinsic interpretation of this measure. Note first that every simple essential geodesic in $\Sigma$ is as simple as possible in $\RO$. This means that the set $\BR_{\ge 0}\cdot \CM\CL_\BZ(\Sigma)$ of multiples of integral measured laminations on $\Sigma$ is contained in $\CQ(A)$. Since the latter is closed we also have that the full space of measured laminations on $\Sigma$ is contained in $\CQ(A)$, that is $\CM\CL(\Sigma)\subset\CQ(A)$. We thus get from the construction \eqref{eq defintion thurston measure} of the Thurston measure $\FM_{\Thu}^\Sigma$ that
\begin{equation}\label{eq bamboleo 1}
\FM_{\Thu}=\lim_{L\to\infty}\frac 1{L^{6g-6+2r}}\sum_{\gamma\in\CM\CL_\BZ(\Sigma)}\delta_{\frac 1L\hat \gamma}
\end{equation}
where, for lack of better notation, we let $\hat\gamma$ be the geodesic representative in $\RO$ of the homotopy class represented by $\gamma\in\CM\CL_\BZ(\Sigma)$. The multicurve curve $\hat\gamma$ is simple in the sense that its lifts to the universal cover $\tilde\RO$ never cross each other. If we denote by $\CM\CL_\BZ(\RO)$ the set of, in this sense, simple geodesic multicurves in $\RO$ then we can rewrite \eqref{eq bamboleo 1} as
\begin{equation}\label{eq bamboleo 2}
\FM_{\Thu}=\lim_{L\to\infty}\frac 1{L^{6g-6+2r}}\sum_{\hat{\gamma}\in\CM\CL_\BZ(\RO)}\delta_{\frac 1L\hat{\gamma}}
\end{equation}
Yet another description of $\FM_{\Thu}$ can be given when we recall that $\RO$ has a finite normal cover (in the category of orbifolds) which is a surface. This means that there are a hyperbolic surface $S$ and a finite group $H$ acting on $S$ by isometries such that $\RO=S/H$. The cover $\pi:S\to\RO$ induces a bijection between elements in $\CM\CL_\BZ(\RO)$ and the set $\CM\CL_\BZ(S)^H$ of $H$-invariant simple multicurves in $S$. %This means that 
%\begin{equation}\label{eq bamboleo 3}
%\FM_{\Thu}=\lim_{L\to\infty}\frac 1{L^{6g-6+2r}}\sum_{\gamma\in\CM\CL_\BZ(S)^H}\delta_{\frac 1L\pi(\gamma)}.
%\end{equation}
The reader familiar with the construction of the usual Thurston measure for surfaces will have no difficulty proving that the limit
$$\FM_{\Thu}^{S,H}\stackrel{\text{def}}=\lim_{L\to\infty}\frac 1{L^{6g-6+2r}}\sum_{\gamma\in\CM\CL_\BZ(S)^H}\delta_{\frac 1L\gamma}$$
exists, where $g$ and $r$ are still the genus and the sum of the numbers of singular points and boundary components of the orbifold $\RO$. Taking into account the identification between $\CM\CL_{\mathbb{Z}}(\RO)$ and $\CM\CL_{\mathbb{Z}}(S)^H$ we get 
$$\FM_{\Thu}=\FM_{\Thu}^{S,H}.$$
Here the left measure lives in the space $\CG(\tilde\RO)$ of geodesics on the universal cover of $\RO$, and the right one lives in the space $\CG(\tilde S)$ of geodesics in the universal cover of $S$, and where the equality makes sense because $\tilde\RO=\tilde S$.

\begin{bem}
It also seems probable that one can recover the measure $\FM_{\Thu}^{S,H}$ as a multiple of the measure on $\CM\CL(S)^H$ obtained by taking a suitable power of the restriction to that subspace of the Thurston symplectic form on $\CM\CL(S)$. It would be interesting to do as in \cite{Monin} and figure out the precise multiple.
\end{bem}

Anyways, the reader having just the present paper in mind can ignore these past comments and just continue thinking of $\FM_{\Thu}$ as given in the proof of Theorem \ref{main}. 

\subsection{Actually counting curves}
Now we prove Theorem \ref{sat1} and Theorem \ref{sat3} from the introduction. Let us start with the latter:

\begin{named}{Theorem \ref{sat3}}
Let $\RO$ be a compact orientable hyperbolic orbifold with possibly empty totally geodesic boundary and let $\CC^{\orb}(\RO)$ be the associated space of geodesic currents. Then the limit
$$\lim_{L\to\infty}\frac 1{L^{6g-6+2r}}\vert\{\gamma\text{ of type }\gamma_0\text{ with }F(\gamma)\le L\}\vert$$
exists and is positive for any $\gamma_0\in\CS^{\orb}(\RO)$ and any positive, homogenous, continuous function $F:\CC^{\orb}(\RO)\to\BR_{\ge 0}$. Here $g$ is the genus of the orbifold $\RO$, $r$ is the sum of the numbers of singular points and boundary components.
\end{named}

The proof of this theorem is the same as that of the analogous result in the case of surfaces \cite{EPS,Book,RS} but let us recap the argument anyways. Other than the fact that $\FM_{\Thu}$ is a Radon measure, and as such locally finite, we will also need that 
\begin{equation}\label{chichu is enjoying the great outside}
\FM_{\Thu}(t\cdot U)=t^{6g-6+2r}\FM_{\Thu}(U)
\end{equation}
for all $U\subset\CC^{\orb}(\RO)$ and all $t\ge 0$. This equality holds true because it does so for the standard Thurston measure $\FM_{\Thu}^\Sigma$ and because the map \eqref{map L geodesics currents} is homogeneous: $\Pi(t\cdot\lambda)=t\cdot\Pi(\lambda)$. Alternatively \eqref{chichu is enjoying the great outside} also follows directly from \eqref{eq bamboleo 2}. Anyways, we are now ready to prove the theorem:

\begin{proof}
Noting that there is nothing to prove if $\RO$ is exceptional, suppose that this is not the case. 

For any such homogenous function $F:\CC^{\orb}(\RO)\to\BR$ we have
\begin{align*}
\frac {\vert\{\gamma\text{ of type }\gamma_0\text{ with }F(\gamma)\le L\}\vert}{L^{6g-6+2r}}
&=\left(\frac 1{L^{6g-6+2r}}\sum_{\gamma\in\Map^{\orb}(\RO)\cdot\gamma_0}\delta_{\gamma}\right)\left(\{F\le L\}\right)\\
&=\left(\frac 1{L^{6g-6+2r}}\sum_{\gamma\in\Map^{\orb}(\RO)\cdot\gamma_0}\delta_{\frac 1L\gamma}\right)\left(\{F\le 1\}\right)
\end{align*}
Now, by Theorem \ref{main} the measures in the last line converge, when $L\to\infty$, to the measure $C(\gamma_0)\cdot\FM_{\Thu}$. Noting that local finiteness of $\FM_{\Thu}$ together with \eqref{chichu is enjoying the great outside} implies that
$$C(\gamma_0)\cdot\FM_{\Thu}\left(\{F=1\}\right)=0.$$
we get that 
$$C(\gamma_0)\cdot\FM_{\Thu}\left(\{F\le 1\}\right)=\lim_{L\to\infty}\left(\frac 1{L^{6g-6+2r}}\sum_{\gamma\in\Map^{\orb}(\RO)\cdot\gamma_0}\delta_{\frac 1L\gamma}\right)\left(\{F\le 1\}\right).$$
Taking all of this together we obtain that
$$C(\gamma_0)\cdot\FM_{\Thu}\left(\{F\le 1\}\right)=\lim_{L\to\infty}\frac 1{L^{6g-6+2r}}\vert\{\gamma\text{ of type }\gamma_0\text{ with }F(\gamma)\le L\}\vert,$$
and we are done.
\end{proof}

We come now to Theorem \ref{sat1}:

\begin{named}{Theorem \ref{sat1}}
Let $\Gamma\subset\PSL_2\BR$ be a non-elementary finitely generated discrete subgroup and $\RO=\BH^2/\Gamma$ the associated 2-dimensional hyperbolic orbifold. Then the limit
$$\lim_{L\to\infty}\frac 1{L^{6g-6+2r}}\vert\{\gamma\text{ of type }\gamma_0\text{ with }\ell_{\RO}(\gamma)\le L\}\vert$$
exists and is positive for any $\gamma_0\in\CS^{\orb}(\RO)$. Here $g$ is the genus of the orbifold $\RO$ and $r$ is the sum of the numbers of singular points and ends.
\end{named}

\begin{proof}
We might once again assume that $\RO$ is not exceptional. Let then $\bar\RO$ be a compact hyperbolic orbifold with interior homeomorphic to $\RO$, consider $\gamma_0\in\CS^{\orb}(\RO)$ as an element in $\CS^{\orb}(\bar\RO)$, and apply Theorem \ref{main} to get
\begin{equation}\label{fefito is a big problem}
\lim_{L\to\infty}\frac 1{L^{6g-6+2r}}\sum_{\gamma\in\Map^{\orb}(\bar\RO)\cdot\gamma_0}\delta_{\frac 1L\gamma}=C(\gamma_0)\cdot\FM_{\Thu}.
\end{equation}
Now note that the same argument that proves it for surfaces shows that there is a compact subset $K\subset\bar\RO\setminus\D\bar\RO$ which contains the geodesic $\phi(\gamma_0)$ for all $\phi\in\Map^{\orb}(\bar\RO)$ (see for example \cite{EPS,Book}). It follows that the measures in \eqref{fefito is a big problem} are all supported by the set $\CC_K^{\orb}(\bar\RO)$ of currents in $\CC^{\orb}(\RO)$ whose support projects to a subset of $K$. Now, as was first proved by Bonahon \cite{BonahonFrench} (see also \cite[Exercise 3.9]{Book}) we have that hyperbolic length function $\ell_{\RO}$ extends continuously to $\CC_K^{\orb}(\bar\RO)$. The claim of Theorem \ref{sat1} follows when we repeat word-by-word the argument in the proof of Theorem \ref{sat3}.
\end{proof}

\section{The key observations}\label{sec:lemmas}

In this section we get the tools needed to prove Proposition \ref{main proposition} in the next section. Notation will be as in the proposition: $\RO=\tilde{\RO}/\Gamma$ is a compact orientable hyperbolic orbifold with possibly empty totally geodesic boundary, 
$$\hat\Sigma=\tilde\RO\setminus\CN_{\hyp}(\sing(\Gamma),\delta)$$ 
is as in \eqref{eq associated surface}, and $\rho$ is the metric on $\hat{\Sigma}$ constructed in Section \ref{sec:metric}.

\subsection{Choosing $\delta$} 
So far, the only condition we have imposed on $\delta>0$ is that it is smaller than $\frac 13\epsilon$ where $\epsilon>0$ satisfies conditions (C1), (C2) and (C3) from Section \ref{sec surface associated to orbifold}. We are momentarily going to give a more stringent condition on $\delta$, but first recall that the convex hull of a connected set $X$ in a negatively curved manifold is the smallest closed connected set $X'$ with the following property: {\em any path in $X$ is homotopic relative to its endpoints to a geodesic path contained in $X'$}. With this language we fix $\delta$ so that, with $\epsilon>0$ as fixed in Section \ref{sec surface associated to orbifold}, the following holds whenever $\gamma\subset\tilde\RO$ is a $\rho_{\hyp}$-geodesic segment:
\begin{equation}
  \tag{C4}\label{star1}
  \parbox{\dimexpr\linewidth-4em}{%
    \strut
   If $r\ge 50\epsilon$ and if $p\in\tilde{\RO}$ is such that $d_{\hyp}(p,\CN_{\hyp}(\gamma,r))\le2\delta$ then the $\rho_{\hyp}$-convex hull of $\CN_{\hyp}(\gamma,r)\cup B_{\hyp}(p,2\delta)$ is contained in $\CN_{\hyp}(\gamma,r)\cup B_{\hyp}(p,\frac\epsilon 2)$.
    \strut
  }
\end{equation}
The lower bound on $r$ guarantees that $\CN_{\hyp}(\gamma,r)$ is uniformly convex. In particular, the existence of such a $\delta$ is evident when one considers the limit case $\delta=0$. In any case, a computation shows that any $\delta<(\frac \epsilon 4)^3$ works. See Figure \ref{fig convex hull} for a schematic representation of (C4).
\medskip

\begin{figure}[h]
\includegraphics[width=0.6\textwidth]{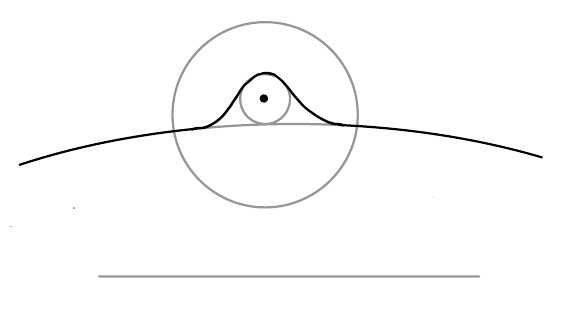}
\caption{Schematic representation of (C4). If a point is close enough to the boundary of the uniformly convex set $\CN_{\hyp}(\gamma,r)$ then its distance to the horizon is really small. The darker line represents the boundary of the convex hull of $\CN_{\hyp}(\gamma,r)\cup B_{\hyp}(p,2\delta)$.}
\label{fig convex hull}
\end{figure}

The reason for imposing (C4) will be clear shortly, but first we need some more notation. If $\gamma\subset\tilde\RO$ is a $\rho_{\hyp}$-geodesic (always compact) segment, consider for $r>0$ the set
\begin{equation}\label{I like pink martini}
\CP(\gamma,r)=\{p\in\sing(\Gamma)\text{ with }d_{\hyp}(p,\gamma)\le r+2\delta\}
\end{equation}
of points which are at most at distance $2\delta$ from the $r$-neighborhood $\CN_{\hyp}(\gamma,r)$ around $\gamma$, and for $t>0$ let 
\begin{equation}\label{a lot}
\CU(\gamma,r,t)=\CN_{\hyp}(\gamma,r)\cup\left(\bigcup_{p\in\CP(\gamma, r)}B_p(t,d_{\hyp})\right)
\end{equation}
be the union of that $r$-neighborhood and the $t$-balls around each point in $\CP(\gamma,r)$. The following lemma gives us, for $t=2\delta$, some control of the convex hull of this set with respect to the metric $\rho$:

\begin{lemma}\label{lem11}
If $\gamma\subset\tilde{\RO}$ is a $\rho_{\hyp}$-geodesic segment then we have
$$\left(\rho\text{-convex hull of }\CU(\gamma,r,2\delta)\cap\hat{\Sigma}\right)\subset\left(\CU\left(\gamma,r,\frac\epsilon 2\right)\cap\hat{\Sigma}\right)$$
for all $r\ge 50\epsilon$. Here $\CU(\gamma,r,t)$ is as in \eqref{a lot}.
\end{lemma}

Before launching the proof recall that convexity of a closed set $X$ is a local property of its boundary. We thus get the following useful property:
\begin{equation}
  \tag{*}\label{star2}
  \parbox{\dimexpr\linewidth-4em}{%
    \strut
    If $(X_0,X_1,\cdots)$ is a countable locally finite collection of closed subsets of a negatively curved manifold with $X_0\cup X_i$ convex for all $i$ and with $X_i\cap X_j=\emptyset$ for all $i\neq j\ge 1$, then $\cup_{i=0}^\infty X_i$ is convex.
       \strut
  }
\end{equation}

\begin{figure}[h]
\includegraphics[width=0.4\textwidth]{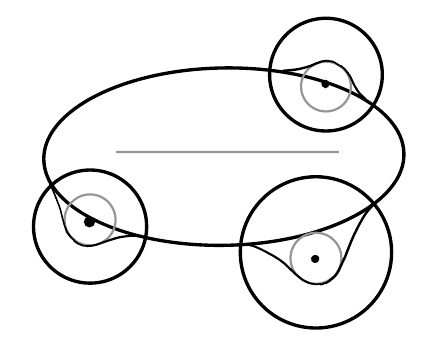}
\caption{Schematic representation of Lemma \ref{lem11}.}
\label{fig neighborhood}
\end{figure}

We now prove the lemma. 

\begin{proof}
Set $X_0=\CN_{\hyp}(\gamma,r)$ and for $p\in\CP(\gamma,r)$ let $Y_p$ be the $\rho_{\hyp}$-convex hull of the union of $\CN_{\hyp}(\gamma,r)\cap\hat\Sigma$ and $B_{\hyp}(p,2\delta)$. Since $r\ge 50\epsilon$ then we get from (C4) that $X_p=Y_p\setminus X_0\subset B_{\hyp}(p,\frac\epsilon 2)$. This implies that the collection of sets $\{X_p\text{ with }p\in\CP(\gamma,r)\}$ is locally finite and that $X_p\cap X_q=\emptyset$ for all distinct $p,q\in\CP(\gamma, r)$. We get thus from (*) that 
$$X=X_0\cup\left(\bigcup_{p\in\CP(\gamma,r)}X_p\right)$$ 
is $\rho_{\hyp}$-convex, meaning that its boundary is $\rho_{\hyp}$-convex. However we have by construction that 
$$\D X\subset\CN_{\hyp}(\gamma,r)\cup\left(\bigcup_{p\in\CP(\gamma,r)}\left(B_{\hyp}(p,\frac\epsilon 2)\setminus B_{\hyp}(p,2\delta)\right)\right),$$
which means that $\D X$ is not only contained in $\hat\Sigma$ but even contained in the part of $\hat\Sigma$ where the metrics $\rho$ and $\rho_{\hyp}$ agree. This means that $\D X$ is not only $\rho_{\hyp}$-convex but also $\rho$-convex. It follows that $X\cap\hat\Sigma$ is a $\rho$-convex set containing $\CU(\gamma,r,2\delta)\cap\hat{\Sigma}$ but contained in $\CU\left(\gamma,r,\frac\epsilon 2\right)\cap\hat{\Sigma}$. The claim follows.
\end{proof}

\subsection{Heights and outgoing rays}

Continuing with the same notation, let $\gamma\subset\tilde{\RO}$ be a $\rho_{\hyp}$-geodesic segment and let $p\in\tilde{\RO}\setminus\gamma$ be a point not on $\gamma$. Under the {\em $\gamma$-outgoing ray at $p$} we understand the $\rho_{\hyp}$-geodesic ray starting at $p$ in the direction of the gradient of the function $d_{\hyp}(\gamma,\cdot)$---that is, the ray that $p$ would follow to escape from $\gamma$ at the fastest possible rate.  

Now let $\eta$ be a $\rho$-geodesic segment in $\hat{\Sigma}$ whose endpoints lie on $\gamma$. The {\em $\gamma$-height of $\eta$}
$$h_\gamma(\eta)=\max\left\{d_{\hyp}(\gamma,p)\ \middle\vert
\begin{array}{l}
\text{ where }p\in\sing(\Gamma)\setminus\gamma\text{ is such that}\\
\text{  the }\gamma\text{-outgoing ray at }p\text{ meets }\eta
\end{array}
\right\}$$
is the maximum hyperbolic distance from $\gamma$ to a cone point $p\in\sing(\Gamma)$ whose $\gamma$-outgoing ray intersects $\eta$---here we take $h_\gamma(\eta)=0$ if we are taking the maximum over the empty set. 

The following lemma asserts that the $\gamma$-height of $\eta$ agrees, up to a small error, with the maximal $d_{\hyp}$-distance to $\gamma$ from points in $\eta$.

\begin{lemma}\label{l:neighborhood} 
Let $\gamma\subset\tilde{\RO}$ and $\eta\subset\hat{\Sigma}$ be a $\rho_{\hyp}$-geodesic segment and a $\rho$-geodesics segment, both with the same endpoints. Then we have
$$\eta\subset\CN_{\hyp}(\gamma,r)$$
where $r=\max\{50\epsilon,h_\gamma(\eta)\}+\epsilon$.
\end{lemma}
\begin{proof}
Let $\CQ$ be the set of those singular points $p\in\sing(\Gamma)$ whose $\gamma$-outgoing ray meets $\eta$. Set $r_0=\max\{50\epsilon,h_\gamma(\eta)\}$ and note that $\CQ\subset\CP(\gamma,r_0)$. The geodesic $\eta$ is homotopic in $\hat\Sigma$ and while fixing its endpoints to a curve contained in
$$\CU(\gamma,r_0,2\delta)\cap\hat{\Sigma}=\left(\CN_{\hyp}(\gamma,r_0)\cup\left(\bigcup_{p\in\CQ}B_{\hyp}(p,2\delta)\right)\right)\cap\hat{\Sigma}.$$
The $\rho$-geodesic $\eta$ is then contained in the $\rho$-convex hull of $\CU(\gamma,r_0,2\delta)\cap\hat{\Sigma}$ and hence in $\CU\left(\gamma,r_0,\frac\epsilon 2\right)\cap\hat{\Sigma}\subset \CN_{\hyp}(\gamma,r_0+\epsilon)$ by Lemma \ref{lem11}. We are done. 
\end{proof}

\subsection{The main observation}

Our next goal is to establish the following fact:

\begin{lemma}\label{l:main}
Let $\gamma\subset\tilde{\RO}$ and $\eta\subset\hat{\Sigma}$ be respectively a $\rho_{\hyp}$-geodesic segment and a simple $\rho$-geodesic segment, such that both segments have the same endpoints $\D\gamma=\D\eta$.  Suppose that at least one of the following holds:
\begin{itemize}
\item[(a)] $h_\gamma(\eta)>1$, or  %\frac 12\log(3)
\item[(b)] $\ell_{\hyp}(\gamma)<\epsilon$ and $h_\gamma(\eta)>50\epsilon$. 
\end{itemize}
Then there is $g\in\Gamma\setminus\Id$ such that $\eta$ and $g\eta$ transversely intersect at least twice. 
\end{lemma} 

\begin{bem}
It follows by a limiting argument and Lemma \ref{l:main} that, if we replace the ``$\max$" in the definition of $h_{\gamma}(\eta)$ by a ``$\sup$", then the lemma also holds when $\gamma$ and $\eta$ are a complete $\rho_{\hyp}$-geodesic and a complete simple $\rho$-geodesic which have the same endpoints in $\partial_{\infty}{\tilde{\RO}}$, the boundary at infinity of $\tilde{\RO}$. To see that this is the case parametrize $\eta:\BR\to\hat\Sigma$, let $\gamma_n$ be the $\rho_{\hyp}$-geodesic segment with endpoints $\eta(-n)$ and $\eta(n)$, set $\eta_n=\eta([-n,n])$, and note that
$$h_\gamma(\eta)\le\liminf_{n\to\infty} h_{\gamma_n}(\eta_n).$$
It thus follows from the lemma that if $h_\gamma(\eta)>1$ then there are $n>0$ and $g\in\Gamma\neq\Id$ such that $\eta_n$ and $g\eta_n$ transversely intersect at least twice. This means a fortiori that $\eta$ and $g\eta$ also meet transversely at least twice.
\end{bem}

\begin{proof}
Starting with the proof of Lemma \ref{l:main}, suppose that $\gamma$ and $\eta$ satisfy one of the two possible conditions in the statement. As a first observation note that if $\gamma$ and $\eta$ satisfy (a) (resp. (b)) and meet in a point other than in their end points, then there are subsegments $\gamma'\subset\gamma$ and $\eta'\subset\eta$ with $\D\gamma'=\D\eta'$, which still satisfy (a) (resp. (b)) and such that $\gamma'$ and $\eta'$ meet only at their endpoints. This means that we can assume without loss of generality that the loop obtained by concatenating $\gamma$ and $\eta$ is simple. Or said differently that $\gamma$ and $\eta$ bound a disk $\Delta$ in $\BH^2$. 

\begin{claim}\label{claim1}
There are a $\rho_{\hyp}$-geodesic segment $\bar\gamma$ and a subsegment $\bar\eta\subset\eta$ satisfying the following properties:
\begin{enumerate}
\item The two segments $\bar\gamma$ and $\bar\eta$ have the same endpoints and disjoint interiors. 
\item The pair $(\bar\gamma,\bar\eta)$ satisfies one of the two conditions (a) and (b) in the statement of the lemma.
\item The disk $\bar\Delta$ with boundary $\bar\gamma\cup\bar\eta$ contains a point $p\in\sing(\Gamma)\cap\bar\Delta$ with $d_{\hyp}(p,\bar\gamma)=h_{\bar\gamma}(\bar\eta)$.
\end{enumerate}
\end{claim}

We suggest that at first, instead of studying the proof of the claim, the reader spends some time looking at Figure \ref{fig producing disk}. 

\begin{proof}[Proof of Claim \ref{claim1}]
If the disk $\Delta$ bounded by the concatenation of $\gamma$ and $\eta$ satisfies (3) then we have nothing to prove. If this is not the case then we will find a hyperbolic geodesic segment $\gamma'$ and a subsegment $\eta'\subset\eta$ satisfying (1) and (2), and such that $\eta'$ is at least $\epsilon$-shorter than $\eta$, that is $\ell_\rho(\eta')\le\ell_\rho(\eta)-\epsilon$. Now, if the disk $\Delta'$ associated to $\gamma'$ and $\eta'$ satisfies (3) then we are done. Otherwise we iterate our procedure... But this process can only be repeated finitely many times because at each step we lose a definite amount of length, and the length of the original segment $\eta$ is finite.

Let us see how we find $\gamma'$ and $\eta'$. We start by taking a point $p\in\sing(\Gamma)$ such that the $\gamma$-outgoing ray $\sigma$ at $p$ meets $\eta$ and with $d_{\hyp}(p,\gamma)=h_\gamma(\eta)$. Note that Lemma \ref{l:neighborhood} implies that all intersections of $\sigma$ and $\eta$ happen in the annulus $B_{\hyp}(p,\epsilon)\cap \hat\Sigma=B_{\hyp}(p,\epsilon)\setminus B_{\hyp}(p,\delta)$. Since the outgoing ray $\sigma$ intersects $\eta$ and we are assuming that $p\notin\Delta$, $\sigma$ must intersect $\eta$ at least twice. We deduce that there is a closed subsegment 
$$\gamma'\subset\sigma\cap(B_{\hyp}(p,\epsilon)\cap \hat\Sigma)$$
with $\gamma'\cap\eta=\D\gamma'$. Let $\eta'\subset\eta$ be the subsegment of $\eta$ bounded by $\D\gamma'\subset\eta$. 

\begin{figure}[h]
\includegraphics[width=1\textwidth]{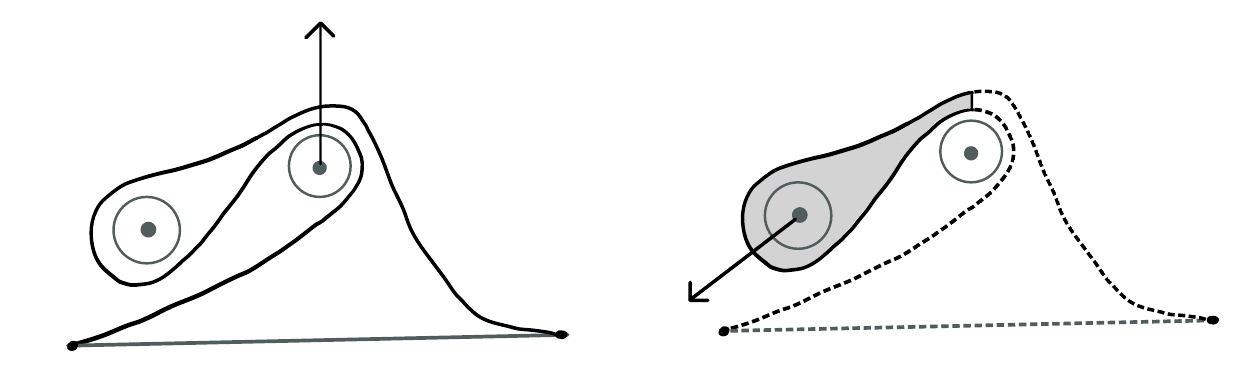}
\caption{Proof of Claim \ref{claim1}.}
\label{fig producing disk}
\end{figure}

By construction the pair $\gamma',\eta'$ satisfies (1). Moreover, since $\eta$ has to travel at least distance $49\epsilon$ to go from $\gamma$ to $B_{\hyp}(p,\epsilon)$ and since the metrics $\rho$ and $\rho_{\hyp}$ agree on $B_{\hyp}(p,50\epsilon)\setminus B_{\hyp}(p,\delta)$ we get that
$$\ell_{\rho}(\eta')\le\ell_{\rho}(\eta)-98\epsilon,$$
which beats our established goal of reducing the length by $\epsilon$ by a proud $97\epsilon$. It just suffices to prove that the pair $(\gamma',\eta')$ satisfies (2), meaning that one of the conditions (a) or (b) holds. Actually, we are going to argue that they satisfy (b). First, $\gamma'$ is shorter than $\epsilon$ by construction. It thus suffices to check that $h_{\gamma'}(\eta')>50\epsilon$. By Lemma \ref{l:neighborhood}  it suffices to prove that $\eta'$ exits $\CN_{\hyp}(\gamma',51\epsilon)$, or even better, that it exits the ball $B_{\hyp}(p,52\epsilon)$. Indeed, since $\eta'$ has both endpoints in $\sigma$, a hyperbolic ray emanating from $p$, since $\eta'$ is contained in $\eta$, a simple geodesic which exists $B_{\hyp}(p,52\epsilon)$ in both directions, and since the metrics $\rho$ and $\rho_{\hyp}$ agree by construction on $B_{\hyp}(p,52\epsilon)\setminus B_p(2\delta,d_{\hyp})$ we get from Lemma \ref{lem:meet rays once} that $\eta'$ cannot be contained in $B_{\hyp}(p,52\epsilon)$. We have proved the claim.
\end{proof}

Continuing with the proof of Lemma \ref{l:main} and with notation as in Claim \ref{claim1} choose $g\in\Stab_{\Gamma}(p)$ with rotation angle $\theta\in[\frac{2\pi}{3}, \frac{4\pi}{3}]$. Note that such $g$ always exists.

\begin{claim}\label{claim2}
We have $g^{\pm 1}(\bar\gamma)\cap\bar\Delta=\emptyset.$
\end{claim}

%\begin{figure}[h]
%\includegraphics[width=0.8\textwidth]{Pics/whatisit.pdf}
%\caption{blah}
%\label{ugly}
%\end{figure}

\begin{proof}[Proof of Claim \ref{claim2}]
Suppose first that the pair $(\bar\gamma,\bar\eta)$ satisfies (a), meaning that 
$$d_{\hyp}(p,\bar\gamma)=h_{\bar\gamma}(\bar\eta)>1.$$
A computation using standard hyperbolic trigonometry implies that 
$$d_{\hyp}(\bar\gamma,g^{\pm 1}\bar\gamma)=2\cosh^{-1}\left(\frac{\sqrt 3}2\cosh(d_{\hyp}(p,\bar\gamma))\right)\ge \frac 32d_{\hyp}(p,\bar\gamma)>h_{\bar\gamma}(\bar\eta)+\epsilon.$$
The claim follows thus because $\bar\Delta\subset\CN_{\hyp}(\bar\gamma,h_{\bar\gamma}(\bar\eta)+\epsilon)$ by Lemma \ref{l:neighborhood}. We are done if the pair $(\bar\gamma,\bar\eta)$ satisfies (a). Suppose now that it satisfies (b) and let $p^*$ be projection of $p$ to $\bar\gamma$. Now, either again a hyperbolic geometry computation, or just plainly comparing with the comparison euclidean triangle one gets that
$$d_{\hyp}(p^*,g^{\pm 1}p^*)\ge\sqrt{ 3}\cdot d_{\hyp}(p,p^*)\ge \sqrt 3\cdot h_{\bar\gamma}(\bar\eta)>h_{\bar\gamma}(\bar\eta)+3\epsilon.$$
Again the claim follows from Lemma \ref{l:neighborhood}.
\end{proof}

We are now ready to conclude the proof of the lemma. First note that 
$$g\bar\Delta\cap\bar\Delta\neq\emptyset$$ 
because $gp=p$. Since $g\bar\Delta$ can neither be contained in $\bar\Delta$ nor contain $\bar\Delta$ we deduce that $\vert g(\D\bar\Delta)\cap\D\bar\Delta\vert\ge 2$. Since $\D\bar\Delta=\bar\gamma\cup\bar\eta$ we get from Claim \ref{claim2} that
$$g\bar{\gamma}\cap(\partial \bar\Delta) =\emptyset \text{ and } \bar{\gamma}\cap g(\partial \bar\Delta)=\emptyset.$$
It follows that $g(\partial \bar\Delta)\cap(\partial \bar\Delta)=g\bar\eta\cap\bar\eta$ and hence that $\vert g\bar\eta\cap\bar\eta\vert\ge 2$. The lemma then follows because $\bar\eta\subset\eta$. 
\end{proof}

\section{Proof of Proposition \ref{main proposition}}\label{sec:proof of proposition}
We are now finally ready to prove the remaining proposition:

\begin{named}{Proposition \ref{main proposition}}
Let $\RO$ be as in the statement of Theorem \ref{main}, $\hat{\Sigma}$ as in \eqref{eq associated surface}, and $\rho$ the metric on $\hat{\Sigma}$ constructed in Section \ref{sec:metric}.  There exists $A\ge 1$ such that any unit speed $\rho$-geodesic $\alpha: \BR\to\hat\Sigma$ which is (1) a quasigeodesic in $(\tilde{\RO},\rho_{\hyp})$ and (2) as simple as possible, is actually $A$-quasigeodesic in $(\tilde{\RO},\rho_{\hyp})$. 
\end{named}

\begin{proof}
Note first that compactness of $\hat\Sigma/\Gamma$, together with $\Gamma$-invariance of the metric $\rho$, implies that the inclusion $(\hat\Sigma,\rho)\hookrightarrow(\tilde\RO,\rho_{\hyp})$ is locally $A_0$-bi-Lipschitz for some $A_0$. It follows in particular that for all $s,t\in\BR$ we have
\begin{equation}\label{eq-tired of this1}
A_0\cdot\vert s-t\vert\ge d_{\hyp}(\alpha(s),\alpha(t)).
\end{equation}
The remaining of the proof is devoted to showing that the other inequality in the definition of quasigeodesic holds for some constant independent of the concrete $\alpha$. 

As a first step we choose a $\Gamma$-invariant triangulation $\CT$ of $\tilde\RO$ whose edges $\CT$ are $\rho_{\hyp}$-geodesic segments of length at most $\epsilon$. Consider the set
$$\CP=\{\text{simplices of }\CT\text{ contained in }\CN_{\hyp}(\sing(\Gamma),60\epsilon)\}$$ 
and let $\vert\CP\vert=\cup_{\sigma\in\CP}\sigma\subset\tilde\RO$ be the union of all those simplexes. Let also 
$$\CE=\{\text{edges of }\CT\text{ contained in the closure of }\tilde\RO\setminus\vert\CP\vert\}$$
and let $\vert\CE\vert=\cup_{e\in\CE}e\subset\tilde\RO$ be the graph obtained by taking the union of all the edges of $\CT$ which are disjoint of the interior of $\vert\CP\vert$. Note that the elements in $\CE$ are not only $\rho_{\hyp}$-geodesic but also $\rho$-geodesic because the metrics $\rho$ and $\rho_{\hyp}$ agree away from a very small neighborhood of $\sing(\Gamma)$. Note also that compactness of $\RO=\tilde\RO/\Gamma$ and $\Gamma$-invariance of $\CT$ imply that there is $C$ such that the following holds for all $p\in\tilde\RO$:
\begin{itemize}
\item[(a)] Less than $C$ elements of $\CE$ intersect the ball $B_{\hyp}(p,4)$.
\end{itemize}
We care about all of this because we will estimate the $\rho$-length of subsegments of $\alpha$ as in the statement of Proposition \ref{main proposition} in terms of the number of edges in $\CE$ that they meet. The key observation is that, since $\alpha$ as in the statement is as simple as possible and hence also simple, and since it cannot spend infinite time in any single ball since it is quasigeodesic, we get from Lemma \ref{lem:meet rays once} that there is some $D>\epsilon$ such that:
\begin{itemize}
\item[(b)] If $\alpha:\BR\to\tilde\RO$ is as in the statement of the proposition and if $[s,t]\subset\BR$ is such that $\alpha[s,t]\cap\vert\CE\vert=\emptyset$ then $\vert s-t\vert<D$.
\end{itemize}

We get from (b) that a geodesic $\alpha$ as in the statement never spends much time without meeting one of the edges in $\CE$. We prove next that once such an $\alpha: \mathbb{R}\to\hat{\Sigma}$ leaves $e\in \CE$ it never comes back to $e$. Indeed, suppose that $s<t$ are such that $\alpha(s),\alpha(t)\in e$ for some $e\in\CE$. Denote by $\gamma$ the subsegment of $e$ between $\alpha(s)$ and $\alpha(t)$ and let $\eta=\alpha[s,t]$. We claim first that $\eta\subset\CN_{\hyp}(\gamma,55\epsilon)$. Otherwise we get from Lemma \ref{l:neighborhood} that $h_\gamma(\eta)>50\epsilon$ and then from Lemma \ref{l:main} that there is $g\in\Gamma$ such that $\vert\eta\cap g\eta\vert\ge 2$, contradicting the assumption that $\alpha$ is as simple as possible. We have thus proved that $\eta\subset\CN_{\hyp}(\gamma,55\epsilon)$. But noting that $\CN_{\hyp}(\gamma,55\epsilon)\subset\CN_{\hyp}(e,55\epsilon)\subset\hat{\Sigma}$ is contractible and that both metrics $\rho$ and $\rho_{\hyp}$ agree thereon we deduce that $\gamma=\eta$ because both are geodesic segments with the same endpoints. We have established the following key fact:
\begin{itemize}
\item[(c)] If $\alpha:\BR\to\tilde\RO$ is as in the statement of the proposition and if $s<t\in\BR$ are such that $\alpha(s),\alpha(t)\in e$ for some $e\in\CE$ then $\alpha[s,t]\subset e$.
\end{itemize}

The reader surely can at this point imagine how (b) and (c) interplay, but might still be wondering why we bothered to state (a) at all. Well, the reason is coming. Still assuming that $\alpha:\BR\to\hat\Sigma$ is a $\rho$-geodesic as in the statement, suppose that $s,t\in\BR$ are such that $d_{\hyp}(\alpha(s),\alpha(t))\le 4$, let $\gamma\subset\tilde\RO$ be the $\rho_{\hyp}$-geodesic segment joining $\alpha(s)$ and $\alpha(t)$, and finally let $p$ be the midpoint of $\gamma$. Since $\alpha$ is as simple as possible we get from Lemma \ref{l:main} that $h_\gamma(\alpha[s,t])\leq1$. Lemma \ref{l:neighborhood} yields then that
$$\alpha([s,t])\subset\CN_{\hyp}(\gamma,1+\epsilon)\subset B_{\hyp}(p,4).$$
We thus get from (a) that $\alpha[s,t]$ meets at most $C$ elements of $\CE$, and from (c) that, if we cut $\alpha[s,t]$ at all the points where we enter an element of $\CE$, then we produce at most $C+1$ segments. Since all of them have at most length $D$ by (b) we obtain:
\begin{itemize}
\item[(d)] If $\alpha:\BR\to\tilde\RO$ is as in the statement of the proposition and if $s,t\in\BR$ are such that $d_{\hyp}(\alpha(s),\alpha(t))\le 4$ then $\vert s-t\vert\le (C+1)\cdot D$.
\end{itemize}

We are almost at the end of the proof of the proposition. Recalling that $\alpha$ as in the statement of the proposition is a quasigeodesic, let $\gamma\subset\tilde\RO$ be the hyperbolic geodesic with the same endpoints as $\alpha$, and let
$$\pi:\BH^2\to\gamma$$
be the nearest point projection. From Lemma \ref{l:main}, or rather from the comment following the said lemma, we get that $h_\gamma(\alpha(\BR))\leq1$. Lemma \ref{l:neighborhood} implies that $d_{\hyp}(\alpha(s), \pi(\alpha(s))\leq 1+\epsilon$ for all $s$. Now, if we have $s,t\in\BR$ with $d_{\hyp}(\pi(\alpha(s)),\pi(\alpha(t)))=1$ we get $d_{\hyp}(\alpha(s),\alpha(t))\le 3+2\epsilon<4$. We thus get from (d) that:
\begin{itemize}
\item[(e)] If $\alpha:\BR\to\tilde\RO$ is as in the statement of the proposition, if $\gamma$ is the $\rho_{\hyp}$-geodesic at bounded distance of $\alpha$, and if $\pi:\tilde\RO\to\gamma$ is the nearest point projection then we have 
$$\vert s-t\vert\le (C+1)\cdot D\stackrel{\text{def}}=A_1$$
for all $s,t\in\BR$ with $d_{\hyp}(\pi(\alpha(s)),\pi(\alpha(t)))\leq 1$.
\end{itemize}
Now, if we have $s<t\in\BR$ arbitrary let $s_0=s$ and, as long as $s_k<t$, define iteratively $s_{k+1}>s_k$ as follows: 
\begin{itemize}
\item If $d_{\hyp}(\pi(\alpha(s_k)),\pi(\alpha(t)))<1$ then set $s_{k+1}=t$.
\item Else $s_{k+1}=\max\{s'\in[s,t]\text{ with }d_{\hyp}(\pi(\alpha(s_k)),\pi(\alpha(s')))=1\}$.
\end{itemize}
In this way we get a sequence 
$$s=s_0<s_1<s_2<\ldots<s_{N-1}<s_N=t$$
where $N$ satisfies 
$$d_{\hyp}(\pi(\alpha(s)),\pi(\alpha(t)))\ge N-1.$$
On the other hand we get from (e) that $\vert s_k-s_{k+1}\vert\le A_1$. This means that 
$$N\ge \frac{\vert s-t\vert}{A_1}.$$
Taking all of this together we have that for any $\alpha:\BR\to\hat\Sigma$ as in the statement of the Proposition we will have
\begin{equation}\label{eq-tired of this2}
d_{\hyp}(\alpha(s),\alpha(t))\ge d_{\hyp}(\pi(\alpha(s)),\pi(\alpha(t)))\ge\frac{\vert s-t\vert}{A_1}-1
\end{equation}
for all $s,t\in\BR$. 

The claim follows then from \eqref{eq-tired of this1} and \eqref{eq-tired of this2} with $A=\max\{A_0,A_1, 1\}$.
\end{proof}

%\vspace{-1cm}
 
\bibliographystyle{plain}
\bibliography{ref-orb}
\end{document}